\documentclass[12pt]{amsart}
\usepackage{amssymb}

\usepackage{epsfig}  		

\newtheorem{thm}{Theorem}[section]
\newtheorem{prop}[thm]{Proposition}
\newtheorem{lem}[thm]{Lemma}
\newtheorem{cor}[thm]{Corollary}

\theoremstyle{definition}
\newtheorem{definition}[thm]{Definition}

\theoremstyle{remark}
\newtheorem{remark}[thm]{Remark}

\newcommand{\nn}{\notag}

\newcommand{\R}{\ensuremath{\mathbb{R}}}

\newcommand{\half}{\frac{1}{2}}

\newcommand{\eps}{\varepsilon}

\newcommand{\til}[1]{\widetilde{#1}}
\newcommand{\ba}[1]{\overline{#1}}

\newcommand{\cN}{\mathcal{N}}
\newcommand{\bX}{\mathbf{X}}

\renewcommand{\d}{\partial}
\renewcommand{\bar}{\overline}

\newcommand{\bxi}{\boldsymbol{\xi}}
\newcommand{\bde}{\boldsymbol{\de}}

\newcommand{\Om}{\Omega}
\newcommand{\om}{\omega}
\newcommand{\si}{\sigma}
\newcommand{\de}{\delta}
\newcommand{\al}{\alpha}
\newcommand{\ga}{\gamma}

\newcommand{\vareps}{\varepsilon}
\newcommand{\pvint}{P.V.\!\!\int}
\newcommand{\ale}{\lesssim}

\newcommand{\dD}{\widehat{D}}
\newcommand{\cO}{\mathcal{O}}
\newcommand{\ov}{{[0, 1]^2}}
\newcommand{\tG}{\widetilde G}
\newcommand{\hP}{\widehat{P}}
\newcommand{\hS}{\widehat{S}}
\newcommand{\eAp}{e^{A}}
\newcommand{\eAm}{e^{-A}}
\newcommand{\bw}{\mathbf{w}}

\newcommand{\intto}{\int_{T_0}^t}
\newcommand{\inttoe}{\int_{T_0}^{T_e}}
\newcommand{\minT}{\min\{T, T_e\}}
\newcommand{\inttoet}{\int_{T_0}^{\min\{T, T_e\}}}
\newcommand{\intso}{\int_{T_0}^s}
\newcommand{\expo}[1]{\exp\left(#1\right)}
\newcommand{\To}{\mathbb{T}}

\begin{document}

\title[No local double exponential growth in hyperbolic flow]{No Local Double Exponential Gradient Growth in Hyperbolic Flow
for the 2d Euler equation}

\author{Vu Hoang}
\address{Rice University, Department of Mathematics-MS 136, Box 1892, Houston, TX 77251-1892}
\email{hoang@math.wisc.edu}
\author{Maria Radosz}
\address{Rice University, Department of Mathematics-MS 136, Box 1892, Houston, TX 77251-1892/Institute for Analysis, Karlsruhe Institute for Technology (KIT)\\
Kaiserstrasse 89, 76133 Karlsruhe (Germany)}
\email{radosz@math.uni-karlsruhe.de}

\begin{abstract}
We consider smooth, double-odd solutions of the two-di\-men\-sio\-nal Euler equation 
in $[-1, 1)^2$ with periodic boundary conditions. This situation is a possible candidate 
to exhibit strong gradient growth near the origin. 
We analyze the flow in a small box around the origin in a strongly hyperbolic regime and prove that the 
compression of the fluid induced by the hyperbolic flow alone is not sufficient to create 
double-exponential growth of the gradient.
\end{abstract}

\date{\today}
 \maketitle

\section{Introduction}

The question whether solutions of the two-dimensional Euler equation in vorticity form
\begin{align}\label{eqEuler}
\om_t + u \cdot \nabla \om &= 0,\quad u=\nabla^{\perp}(-\Delta)^{-1}\om
\end{align}
($\nabla^{\bot} = (-\d_{x_2}, \d_{x_1})$) can exhibit strong gradient growth in time is a topic of ongoing interest. The best known upper bound predicts
double-exponential growth in time:
\begin{align}\label{doubleExpUpperBound}
\|\nabla \om(\cdot, t)\|_{L^\infty(\Omega)} \leq C_1 \exp(C_2 \exp(C_3 t))
\end{align}
on a domain $\Omega$ with either a smooth boundary with no-flow boundary condition or no boundary (e.g. a torus).
The constants $C_i$ depend on the initial data. 
A natural and important question is: Are there flows for which this upper bound is attained? 
For domains with boundary, a recent breakthrough by A. Kiselev and V. \v Sver\'ak 
\cite{KiselevSverak} answers the question affirmatively. In \cite{KiselevSverak}, solutions are constructed that attain the
double-exponential bound \eqref{doubleExpUpperBound}.

For smooth solutions on the torus, the situation is far from clear.
The best known result  so far was given by S. Denisov.  In \cite{Denisov1}, he shows that at least superlinear gradient growth is possible
and in \cite{Denisov2} he provides an example of double-exponential growth for an arbitrarily long, but finite
time interval. In the recent paper \cite{Zlatos}, A. Zlato\v s constructs initial data leading to exponential gradient growth, 
his solution is however in $C^{1, \gamma}$ for some $\gamma \in (0, 1)$ and not in $C^2$. 

In \cite{KiselevSverak} the construction is based on creating a \emph{hyperbolic flow scenario}.  By imposing a symmetry 
on the solutions, a stagnant point of the flow is created on the boundary of the domain. The initial conditions are chosen
in such a way the flow on the boundary is directed towards the stagnant point, creating a strong fluid compression
and therefore strong gradient growth.

A natural way to carry the Kiselev-\v Sver\'ak construction to the torus is to consider double-odd solutions, i.e.
\begin{align}\label{doubleOddSym}
\om(-x_1, x_2)& = - \om(x_1, x_2), ~\om(x_1, -x_2) = - \om(x_1, x_2),
\end{align}
This construction was employed in \cite{Zlatos}. In \cite{Denisov2}, a perturbation argument
starting from a non-smooth double-odd stationary solution (see \cite{Bahouri}) was used. So far, however, creating infinite-time 
\emph{double-exponential} growth in the double-odd scenario was not succesful. Our goal in this paper is to 
explore the difficulties in using this scenario, by proving a conditional regularity result.

It is interesting to notice that the result \cite{KiselevSverak} is in some sense analogous to the still open blowup problem 
for for the more singular surface quasigeostrophic equation. In SQG blowup means that the solution becomes singular in finite time whereas for the 2d Euler 
equation ``blowup'' would mean maximal (double-exponential) gradient growth on an infinite time interval. 
There are important conditional regularity results for the SQG equation such as \cite{Cordoba,CordobaFefferman}, where 
the authors study a certain blowup scenario, in order to finally exclude it.
An analogous ``conditional regularity result'' for 2d Euler equation would be to show that in certain scenarios maximal gradient growth does not occur. Since 
the possible motions of fluids are various and in general very complicated, studying scenarios is an invaluable method to gain insight into 
regularity problems of fluid mechanics.

Our main result states that  a hyperbolic flow cannot create double-exponential gradient growth near the origin by itself when we 
start with double-odd $C^2$ initial data, provided a certain ``upstream" control is assumed on the flow.
This is an important step into understanding the double-odd hyperbolic scenario since we rule out the most promising 
candidate for a mechanism creating maximal gradient growth, i.e. the local hyperbolic compression. Our result does not imply
impossibility of double-exponential growth in general, but makes the construction of examples much harder.  

In some sense, the scenario considered here is complementary to the one considered by D. Cordoba for the SQG
equation in \cite{Cordoba}, where a closing hyperbolic saddle is considered. There the solution stays smooth except for
the possible closing of the saddle. In our scenario for 2d Euler, the hyperbolic saddle is fixed due to the symmetry
($\om=0$ on the coordinate axes), and we are asking if blowup can happen in another way.

We strongly believe that the techniqes developed here will also be useful in understanding the hyperbolic
scenario for other models in fluid mechanics and also in situations with a physical boundary. There, although the goal
is to prove the existence of a blowup, a certain amount of control up to the blowup time is necessary.   

Interesting results concerning the related question of existence of double-exponential growth in the context
of (nonsmooth) patch solutions were given by S. Denisov (see \cite{Denisov3}).

Finally, we would like to mention the recent preprint \cite{Katz}, where a different approach is proposed to study whether double-exponential gradient growth can occur at an interior point. 

\subsection{Setup and feeding conditions.}
We consider \eqref{eqEuler} on $\To=[-1,1)^2$ with periodic boundary conditions and double-odd $C^2$ initial data $\om_0$. 
From now on, we use $\|\cdot\|_\infty$ to denote the $L^\infty$-norm on the torus $\To$.

The double-odd symmetry is preserved by the evolution and \eqref{doubleOddSym} implies that the origin is a stagnant point 
of the flow field for all times. Moreover, the flow on each coordinate axis is always directed along that axis. When considering 
smooth solutions $\om\in C^1([0, \infty), C^2(\To))$,
\eqref{doubleOddSym} also implies 
\begin{align*}
\om = 0
\end{align*} 
on the coordinate axes.

We will study the flow in boxes of the form
\begin{align*}
D&=(0,\de_1)\times(0,\de_2),~~\dD=(0,\de_1+\de_3)\times(0,\de_2),
\end{align*}
where $\de_j$ are positive, but small and  
\begin{align*}
0<\de_1< \de_2 < \de_1+\de_3. 
\end{align*}
In a hyperbolic flow,  the origin is a stagnant point of the flow and fluid particles constantly enter the box
$D$ from the right and leave on the top (see Fig.~\ref{fig1}). The particles moving on the $x_1$-axis approach
asymptotically the origin and never leave the box $D$.
Generally speaking, there is, a compression of the fluid in $x_1$-direction
and a decompression in $x_2$-direction (or the other way around). The vorticity is zero on the axes. The gradient growth
in the box $D$ comes from two sources: particles that were at $t=0$ inside $D$ and those which enter the box at later times.
The time evolution of the gradient of those particles entering the box is difficult to control over infnite times, and is generated by
flow situations which have little to do with the hyperbolic scenario. We are interested in making local statements and must \emph{assume} 
a certain control on the flow entering the box $D$. 

We shall therefore call $\dD\setminus D$ \emph{feeding zone} and formalize this idea in the following definition
(the meaning of the parameter $\al$ will become clear later).
\begin{definition} Let $\al\in(0,\frac{1}{4})$.
The box $\dD$ is said to satisfy the conditions of {\em controlled feeding}, with feeding parameter $R \geq 0$ if 
\begin{align}\label{eqfeeding1}
\left|\frac{\d \om}{\d x_1}(x, t)\right|\leq R x_2^{1-\al}, ~~\left|\frac{\d \om}{\d x_2}(x, t)\right|\leq R \quad (x\in \dD\setminus D)
\end{align} 
for all times $t \geq 0$.
\end{definition}
We can think of the first inequality in \eqref{eqfeeding1} as a H\"older-version of a bound on $\d_{x_2, x_1}\om$, keeping in mind that $\d_{x_1}\om(x_1, 0, t)=0$
for all times. The concept of controlled feeding conditions allows us to study the evolution of $\om$
in $D$ independent of the remaining flow. Note that for the purposes of this paper, we consider time-independent $R$ only
(see also Remark \ref{timeDepFeeding}).

\subsection{The hyperbolic scenario.} \label{secHyperbolic}

In order to give a definition of hyperbolic flow suitable for our purposes, we introduce the following
important quantity. Let $\al\in (0, \frac{1}{4})$ be fixed. For a smooth, periodic function $\om$ we set
\begin{align*}
M(x, t) := \max_{0\le y_1, y_2\le \max\{x_1, x_2\}}\left\{\left|y_1^\al \frac{\d \om}{\d x_1}(y, t)\right|, \left|y_2^\al \frac{\d \om}{\d x_2}(y, t)\right|\right\}+\|\om\|_\infty.
\end{align*}
Note that $M(x,t)$ also depends on $\om$ and $\al$.
The velocity field $u(x, t) := \nabla^\bot (-\Delta)^{-1}\om$ for double-odd $\om$ 
($\om$ with mean zero over $\To$) can be written in the form
\begin{align}\label{velStruc}
u_1(x, t) = - x_1 Q_1(x, t), ~u_2(x, t) = x_2 Q_2(x, t)
\end{align}
where $Q_1, Q_2$ are scalar fields given by certain integral operators (see \eqref{defQ}) acting on $\om$.
The following definition states that we regard the flow as hyperbolic if both $Q_1$ and $Q_2$ essentially have a positive
lower bound, up to a term controlled by the quantity $M(x,t)$. 

\begin{definition}\label{defhyperbolicflow}
Let $\om$ be a smooth solution of the Euler equation, and let $\al\in (0, \frac{1}{4})$ be fixed.
We say that the flow is hyperbolic near the origin if there are constants $\rho, A, \beta_0>0$ 
for which the following condition is satisfied for all $t \in [0, \infty)$:
\begin{align}\label{Hyperbolicity}
Q_i(x, t) + A |x|^{1-\al} M(x, t) \geq \beta_0 > 0 \qquad(0\le x_1, x_2 \le \rho,~i=1,2).
\end{align}
\end{definition}
The model situation for a hyperbolic flow is the following: Consider the dynamical system
\begin{align*}
\dot x_1 &= - A x_1, ~~ \dot x_2 = B x_2
\end{align*}
with positive constants $A, B$. This system has a hyperbolic saddle point in $(0, 0)$, which
is a stagnant point of the flow. On the axes, the flow is directed along the axes.
 
The velocity field given by \eqref{velStruc} generalizes this structure if $Q_i(x, t) \geq \beta_0 > 0$. A further 
generalization necessary for our result is \eqref{Hyperbolicity}, since the stronger condition $Q_i(x, t) \geq \beta_0 > 0$
cannot be easily realized. In certain flow situations, the term $|x|^{1-\al} M(x, t)$ is small close to the origin. 
Bounding the quantity $M(x, t)$ plays a central role in our estimates. 

\begin{remark}
By choosing the initial data $\om_0$  suitably, we can ensure hyperbolic flow.
One possible choice is, for example, choosing $\om_0$ to be nonnegative in $[0, 1]^2$ and such that $\om_0 = 1$ on a set of sufficiently large measure, as it
was done in \cite{KiselevSverak, Zlatos}. This creates 
a situation where \eqref{Hyperbolicity} is satisfied (see theorem \ref{thmLowerboundQ2}). 
In this sense, \eqref{Hyperbolicity} is a ``realistic" condition on the flow. 
\end{remark}

\subsection{Main result.}
Our main result is the following theorem. 
\begin{thm} \label{main} Fix $0< \al < \frac{1}{4}, 0<\de_3< 1/2$. Let $\om$ be a $C^2$, double-odd solution of the Euler equation with initial data $\om_0$, and  
suppose the flow is hyperbolic near the origin.  Let $R> \|\om_0\|_\infty > 0$ be given. 
Then there exist small $\de_1, \de_2>0$ depending on $\al, \beta_0, R$ and $\om_0$, such that if $\dD$ satisfies the controlled feeding conditions with parameter $R$,
then  
\begin{align*}
\|\nabla \om\|_{D, \infty} \leq C_1 \exp( C_2 t)\quad (t\in [0, \infty))
\end{align*}
for some $C_1, C_2 > 0$ depending on $R, \al,\beta_0, \de_1,\de_2, \de_3, \om_0$.
\end{thm}
This means that the hyperbolic compression alone and the interaction of the fluid inside the box is not sufficient
to create double-exponential gradient growth.
One would have to create a scenario where the feeding conditions are violated. This means roughly that there has to be a kind of compression in
$x_2$-direction in the feeding zone. This would have to be caused by much more complicated interactions outside the box. At the present
time, no such scenario is known.

\section{Gradient growth in the hyperbolic scenario}

Before describing our approach, let us explain first why at first sight the hyperbolic scenario seems to be a good candidate 
for double-exponential growth. Namely, for $Q_1, Q_2$ we have the upper bounds
\begin{align*}
Q_1(x, t), Q_2(x, t) \ale \|\om\|_{\infty} |\log(x_1^2 + x_2^2)|.
\end{align*}
If it were possible to create a situation where a \emph{lower bound} of roughly the same order holds, i.e.
$Q_1 \geq C |\log(x_1^2 + x_2^2)|$ over an infinitely long time interval, then for the particle trajectories 
lying on the $x_1$-axis (i.e. $X_2=0$) 
\begin{align*}
X_1(t) \le \exp( - C_1 \exp( C_2 t))
\end{align*}
would hold, as seen by solving the ODE $\dot X_1 = -X_1 Q_1$. If, moreover, one could arrange for the initial data $\om_0$ 
to have suitable nontrivial values on the $x_1$-axis, then this would create double exponential gradient growth. 
However,  the simultaneous requirements of smoothness and double-odd symmetry of $\om$, 
necessarily imply $\om=0$ on the axes. Moreover, it is highly unclear how a such strong lower bound on $Q_1$ could be achieved. 
As we shall see later, a certain amount of smoothness of $\om$ and the vanishing of $\om$ on the axes 
lead to a better upper bound, {\em without} the logarithmic behavior which is crucial for the double-exponential growth. 

Another way one might hope to get double exponential growth is to consider a ``projectile'', i.e. to track the movement
of a small domain close to the origin on which $\om = 1$, as it was done in \cite{KiselevSverak}. There the self-interaction of the projectile
was able to create enough growth in the values of $Q_1$ to allow double-exponential growth. While the projectile
approaches the origin, the values of $Q_1$ on it get larger, this fact being connected to a certain logarithmically divergent integral. 
Our Theorem \ref{main} shows that in general this is not possible for double-odd solutions, unless there is some kind of compression in $x_2$-direction in the feeding zone. Thus a scenario with maximal gradient growth must be much more complicated than using the self-interaction of the projectile.

In fact, provided the feeding condition holds, the steady fluid compression guaranteed by \eqref{Hyperbolicity} will turn out 
to stabilize the flow in the neighborhood of the origin.  That is, the hyperbolicity condition \eqref{Hyperbolicity}
- which is essentially a lower bound on $Q_i$ - is converted in the proof of Theorem \ref{main}
into an upper bound for $Q_i$. This is what finally leads to a bound on the gradient growth in $D$. 


\subsection{Heuristic considerations.}\label{subsec:heuristic}
We now present an intuitive discussion of our result. Fluid particles carried by the hyperbolic flow will
 constantly enter the box $D$ from the right and leave on
the top (see figure \ref{fig1}). All particles except for those moving on the axes spend a finite time
in the box. The particles on the $x_1$-axis move towards the
left approaching the origin asymptotically as $t\to \infty$. Particle trajectories
$t\mapsto \bX(t)=(X_1(t), X_2(t))$ for which $X_2(0)$ is small approximate the
straight trajectories of the particles on the $x_1$-axis for a long time, before going steeply
upward. The time a particle spends in $D$ goes to infinity as $X_2(0)\to 0$.

\begin{figure}[htbp]
\begin{center}
\includegraphics[scale=1]{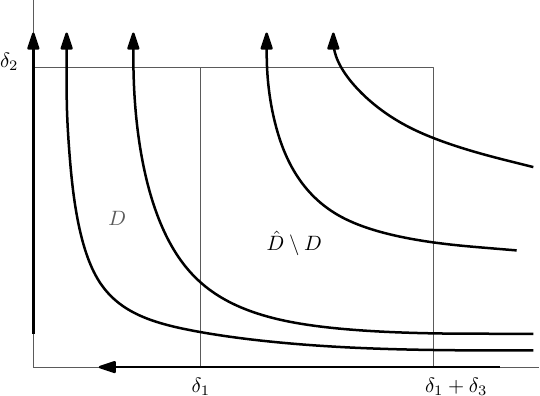}%
\end{center}
\caption{Illustration of a hyperbolic flow.} \label{fig1}
\end{figure} 

We now consider the trajectory of a particle $\bX$. The particle may 
have started inside $D$ at time $t=0$, or may have entered the box at some
time $T_0 > 0$, in which case $\bX(T_0)\in \d D$. Also, assume that the particle
exits the box $D$ at some time $T_e$, i.e. $X_2(T_e)=\de_2$. 
The evolution of the gradient of $\om$ along the trajectory is given by an ODE of the form
\begin{align}\label{ode1}
\frac{d}{d t}\nabla \om(\bX(t), t) = (-\nabla u(\bX(t), t))^T \nabla \om(\bX(t), t)
\end{align}
where $\nabla u$ is the velocity gradient. The relation \eqref{ode1} is simply derived by
differentiating the Euler equation. The key is now to use the structure \eqref{velStruc}
of the velocity field such that we obtain
\begin{align}\label{ode2}
\frac{d}{d t}\nabla \om(\bX(t), t) = \begin{pmatrix}
Q_1 + x_1 \frac{\d Q_1}{\d x_1}& -x_2\frac{\d Q_2}{\d x_1}\\
x_1 \frac{\d Q_1}{\d x_2} & - Q_2 - x_2 \frac{\d Q_2}{\d x_2}
\end{pmatrix}
\nabla \om(\bX(t), t)
\end{align}
We write the right hand side of \eqref{ode2} as 
\begin{equation*}
\begin{pmatrix}
a(t) & c(t) \\
b(t) & -a(t)
\end{pmatrix}
\nabla \om(\bX(t), t).
\end{equation*}
evaluating all matrix entries along the given trajectory $\bX$. Note that the matrix has trace zero, since the
velocity field $u$ is divergence free.
There are several ways we can heuristically regard \eqref{ode2} as a perturbation of an easier problem.

\begin{itemize}
\item For the discussion assume that $Q_1,Q_2>0$ and we can control the derivatives $\d_{x_j} Q_i$ for small $x$. Since in a sufficiently small box $x_1 \d_{x_1} Q_1, x_2 \d_{x_2} Q_2$ should be rather ``small" (due to the prefactors $x_1, x_2$), 
$a$ should be positive and bounded away from zero along the hyperbolic trajectory.
To gain some insight, we consider the case of a particle moving close to the $x_1$-axis, i.e. with small $X_2(T_0) > 0$. We expect that
$c = -x_2 \d_{x_1} Q_2, b = x_1 \d_{x_2} Q_1$ are ``small". Life would be easy if we could neglect $b, c$ and set $b, c=0$ in \eqref{ode2},
so that we have a diagonal system. Denoting $\bxi(t) = \nabla \om(\bX(t),t)$
the solution can be explicitly computed to be 
\begin{align}\label{eqSol}
\xi_1(t) = e^{A(t)} \xi_1(T_0), ~~\xi_2(t) = e^{-A(t)} \xi_2(T_0).
\end{align}
where $A(t) = \int_{T_0}^t a(\bX(s))~ds$. \eqref{eqSol} shows that, in general,
the gradient in $x_1$-direction grows along the particle trajectory.
However, there is
an effect which allows us to cancel the growing factor $e^A$. Assume for the sake of the discussion that the following stronger feeding conditions hold:
\begin{align}\label{eqfeeding2}
\left|\frac{\d \om}{\d x_1}(x, t)\right|\leq R x_2, ~~\left|\frac{\d \om}{\d x_2}(x, t)\right|\leq R,
\end{align}
on $\dD$ for $t=0$ and on $\dD\setminus D$ for all $t>0$.
These imply in either case $T_0=0$ and $T_0>0$
\begin{align}\label{eqIntro1}
|\xi_1(t)|\leq e^{A(t)} |\xi_1(T_0)| = e^{A(t)}\left|\frac{\d \om}{\d x_1}(\bX(T_0), T_0)\right| \leq R e^{A(t)} X_2(T_0).
\end{align}
Now we observe that
\begin{align}\label{eqIntro3}
A(t) \approx \intto Q_2(s)~ds
\end{align}
temporarily neglecting the ``small'' term $x_2 \d_{x_2} Q_2$. Now from \eqref{velStruc} we have the differential equation $\dot X_2 = X_2 Q_2$, so that 
$X_2(t) = X_2(T_0) \expo{\intto Q_2(\bX(s)) ~ds}$
and hence
\begin{align}\label{eqIntro2}
X_2(T_0) = X_2(T_e) \expo{- \inttoe Q_2(\bX(s)) ~ds}\leq \de_2\expo{- \inttoe Q_2(\bX(s))~ ds}.
\end{align}
Combining \eqref{eqIntro2}, \eqref{eqIntro1} and \eqref{eqIntro3}, we get
\begin{align*}
|\xi_1(t)|\leq \de_2 R \expo{ - \int_{t}^{T_e} Q_2(\bX(s))~ds} \leq \de_2 R,
\end{align*}
suggesting that the gradient
in $x_1$-direction \emph{does not grow at all in time} given the feeding condition \eqref{eqfeeding2}. Our rigorous result does not give such a strong conclusion, but we will be able to prove that the gradient 
\emph{grows at most exponentially in time} using a weaker feeding condition. In Remark \ref{remAlpha} we explain why \eqref{eqfeeding2} is not an appropriate feeding condition for the problem. 
\end{itemize}

The heuristics appear deceivingly simple, but in order to make the argument rigorous, we have to overcome a number of 
formidable technical difficulties. To begin with, the coefficients of \eqref{ode2} depend on the solution $\om$ through the integral operators $Q_1, Q_2$. 
The derivatives $\d_{x_1} Q_1, \d_{x_2} Q_2$ are given by singular integral operators. 

Of course, none of the coefficients may be neglected, and we have to produce sufficiently good estimates on the solutions of the full ODE system \eqref{ode2}. A major obstacle in getting good estimates, however, is caused by the unstable nature of \eqref{ode2}. To illustrate this we consider a tridiagonal system by setting $c=0$, but keeping $b$, so
that we get a supposedly better approximation than the diagonal system. In this model, too, the solutions can be calculated explictly, and we get
\begin{align*}
\xi_1(t) = e^{A(t)} \xi_1(T_0), ~~\xi_2(t) = e^{-A(t)} \left[ \xi_2(T_0)+ \xi_1(T_0)\int_0^t b(s) e^{2 A(s)}~ds\right].
\end{align*}
This shows that not only the derivative in $x_1$-direction but also the derivative in $x_2$-direction of $\om$ may potentially 
grow in time (due to the contribution $e^{-A(t)}\int_0^t b(s) e^{2 A(s)}~ds$). To make things worse, a possible strong growth in $\d_{x_2}\om$ is coupled back into the coefficients of
the ODE \eqref{ode2} via our estimates on $\d_{x_1} Q_1, \d_{x_2} Q_2$. On the other hand, by a similar argumentation as in the case of the diagonal system,
the factor $\xi_1(T_0)$ may help via a feeding condition. 
We need therefore to proceed with extreme care, looking to cancel the growing factor $e^A$ with the decaying factor $e^{-A}$ whenever possible.

\begin{remark}\label{timeDepFeeding}
In our scenario, we always assume the intensity of the feeding (i.e.~the quantity $R$) to be time-independent. 
One might think of allowing the feeding parameter to grow in time to include more complicated scenarios. 
However, this is met with considerable challenges.

Firstly, it is not clear what a realistic condition on $R$ should be, since it depends on the complexity of the flow away from the origin.
One concrete situation where we can imagine a reasonable time-dependent feeding condition is as follows:  
A vortex created by a large patch (see Figure \ref{fig3}) where $\om$ is constantly $1$. 
The flow revolves in clock-wise direction around the patch and,
in analogy with a shear flow, one could assume linear growth in time of the gradient in the feeding zone. 

The application of the techniques developed here to time-dependent feeding are not straightforward
(see Remark \ref{aboutTimeDepFeeding}), due to the non-local and non-linear nature of the problem. 
\end{remark}
\begin{figure}[htbp]
\begin{center}
\includegraphics[scale=1]{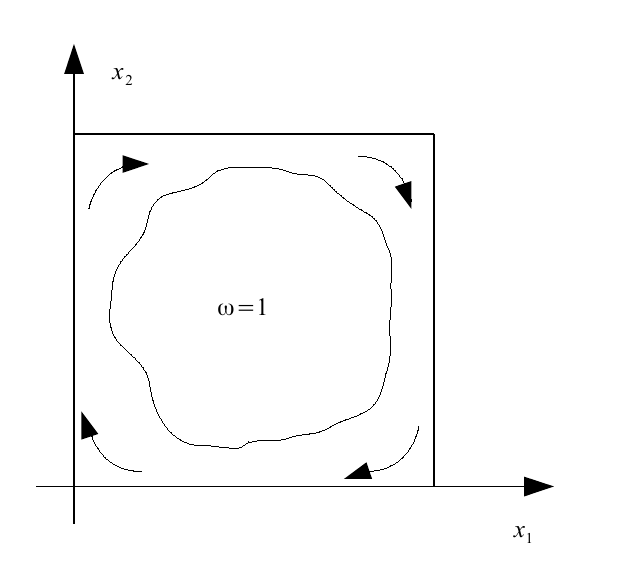}%
\end{center}
\caption{Flow around a patch.} \label{fig3}
\end{figure} 

\section{Notation}

\subsection{Euler velocity field.}

For $x=(x_1,x_2)$ we write $\til x = (-x_1,x_2)$ and $\bar x=(x_1,-x_2)$. The velocity field for the Euler equation is
\begin{align*}
u(x, t) = \frac{1}{2\pi}\int_{\R^2} \frac{(y-x)^\bot}{|y-x|^2} \om(y, t)~dy
\end{align*}
where $\om\in C^2(\To)$ is periodically extended to all of $\R^2$ and $z^\bot = (-z_2, z_1)$.
In the calculation of the integral a limit in the
mean (sequence of unboundedly growing domains) is understood. Note that the velocity field is $\nabla^\bot (-\Delta)^{-1}\om$,
where $-\Delta$ is the periodic Laplacian on the Torus $\To$. A simple calculation using the double-odd symmetry of $\om$ leads to
\begin{align*}
u_1(x, t) = -x_1 Q_1(x, t), ~~u_2(x, t) = x_2 Q_2(x, t)
\end{align*}
where $Q_1, Q_2$ are the following integral operators (see Appendix B)
\begin{align}\label{defQ}\begin{split}
Q_1(x, t)&=c_0\int_\ov[G_1^1(x,y)+G_1^2(x,y)]\om(y,t)~dy+Q_1^r(x, t)\\
Q_2(x, t)&=c_0\int_\ov[G_2^1(x,y)+G_2^2(x,y)]\om(y,t)~dy+Q_2^r(x, t)
\end{split}
\end{align} 
with kernels
\begin{align*}
&G_1^1(x,y)=\frac{y_1(y_2-x_2)}{|y-x|^2|y-\til x|^2},
&G_1^2(x,y)=\frac{y_1(y_2+x_2)}{|y+x|^2|y-\bar x|^2},\\
&G_2^1(x,y)=\frac{y_2(y_1+x_1)}{|y+x|^2|y-\til x|^2},
&G_2^2(x,y)=\frac{y_2(y_1-x_1)}{|y-x|^2|y-\bar x|^2},
\end{align*}
where $c_0$ denotes the right constant.
The expression $Q_1^r$ is given by the following (limit in the mean) integral
\begin{align*}
c_0 \int_{\R^2_+ \setminus \ov}[G_1^1(x,y)+G_1^2(x,y)]\om(y)~dy,\qquad 
\R^2_+ = (0, \infty)^2,
\end{align*}
a similar formula holding for $Q_2^r$. 

In section \ref{PotentialTheory} we will derive estimates for the entries of the matrix in \eqref{ode2} which are 
independent of the trajectory. For this purpose it is convenient to use the definitions:
\begin{align}\label{def:abc}
\begin{split}
&a(x,t):=Q_1(x,t)+x_1\frac{\d Q_1}{\d x_1}(x,t)=Q_2(x,t)+x_2\frac{\d Q_2}{\d x_2}(x,t),\\
&b(x,t):=x_1\frac{\d Q_1}{\d x_2}(x,t),\\
&c(x,t):=-x_2\frac{\d Q_2}{\d x_1}(x,t).
\end{split}
\end{align}
Moreover, since the estimates will be for fixed $t$ we shall often skip the $t$ variable in the notation. 
When evaluating $a,b,c,Q_i$ etc.~along a particle trajectory $\bX(t)$ in section \ref{sec:main} we shall write
$a(t):=a(\bX(t),t)$ etc. reconciling with the notation in section \ref{subsec:heuristic}.

\subsection{Convention for estimates.} The notation $f \ale g$ means
\begin{align*}
f \le C g,
\end{align*}
where $C$ may depend on $\al, \beta, \|\om\|_\infty$ and on universal constants,
e.g. geometrical characteristics of the domain $\To$.  $C$ does not depend on $\de_1, \de_2, \de_3,t$.
When using this notation, we shall always imply that $C<\infty$ for all $\al\in (0, \frac{1}{4})$.

\section{Potential theory of $Q_1, Q_2$} \label{PotentialTheory}

\subsection{Sufficient conditions for hyperbolic flow}

We will be working with boxes of the form
\begin{align}
D&=(0,\de_1)\times(0,\de_2)\nn \\
\dD&=(0,\de_1+\de_3)\times(0,\de_2) \label{defD},
\end{align}
with the following restriction:
\begin{align}\label{condDelta}
0<\de_1 < \de_2 < \de_1+\de_3.
\end{align}
and $\de_j$ so small that $\dD\subset [0, 1]^2$. We also write
\begin{align*}
d(x) = \de_2 - x_2
\end{align*}
which is the distance of the point $x$ to the top of the box. We write $\bde=(\de_1, \de_2), |\bde|^2=\de_1^2+\de_2^2$.

Define
\begin{align*}
M_{D}(t) := \max_{y\in D}\left\{\left|y_1^\al \frac{\d \om}{\d x_1}(y,t)\right|, \left|y_2^\al \frac{\d \om}{\d x_2}(y,t)\right|\right\}+\|\om\|_\infty
\end{align*}
and $M_{\dD}$ for the analogous quantity.
Note that $M_{D}$ and $M_{\dD}$ depend on $\om$ and $\al$. 


As mentioned before, the flow near the origin can be made hyperbolic, with compression 
in the $x_1$-direction and expansion in $x_2$-direction by choosing the initial data such that
$\om_0 \geq 0$ on $\ov$ and such that
\begin{align*}
\mathfrak{m} := |\{ x : \om_0(x) = \|\om_0\|_\infty \}|
\end{align*}
is sufficiently large. This is a consequence of theorem \ref{thmLowerboundQ2}.

\begin{remark}
\begin{enumerate}
\item [(a)] As a consequence of $\om=0$ on the coordinate axes we have the following important inequality
\begin{align}\label{omegaIneq}
|\om(y,t)|\ale M(x,t)\, y_j^{1-\al}\qquad(y_1,y_2\le \max\{x_1,x_2\} )
\end{align}
where $j=1,2$.
\item[(b)] The periodicity and double-oddness of $\om(\cdot, t)$ imply also the reflection symmetries
\begin{align*}
\om(1+x_1, x_2, t) = -\om(1-x_1, x_2, t), ~\om(x_1, 1+x_2)=-\om(x_1, 1-x_2).
\end{align*}
Consequently, the four corner points of $[-1, 1]\times [-1, 1]$ are also stagnant points
of the flow, the flow being confined in $\ov$. Hence $\om_0 \ge 0$ on $\ov$ implies $\om(x, t)\ge 0$ on $\ov$ for all times,
a fact we shall use below.
\end{enumerate}
\end{remark}

\begin{thm}\label{thmLowerboundQ2}
Suppose $\om_0(x)\ge 0$ on $\ov$. There exist universal $0<m_0<1$ and $0 < K$ such that if $m_0<\mathfrak{m}<1$, there are 
$\beta_0 > 0, A>0$ such that the following estimates hold for all times
\begin{align}\label{ineqthmLowerboundQ2}
\begin{split}
Q_2(x, t) + A M(x, t) |x|^{1-\al} \geq \beta_0 \\
Q_1(x, t) + A M(x, t) |x|^{1-\al} \geq \beta_0  
\end{split}
\end{align}
for $|x| \leq K (1-\mathfrak{m})$, i.e. the flow is hyperbolic near the origin.
\end{thm}

To prove this, we need the following Lemma, which is an adaption of a result in \cite{Zlatos}.
\begin{lem}
Let $\Om(2x) := [2 x_1, 1]\times[2 x_2, 1]$. Suppose $\om(x)\ge 0$ for $x\in \ov$. Then the estimate
\begin{align*}
Q_i(x,t) \geq c_0 \int_{\Om(2x)} \frac{y_1 y_2}{|y|^4}\om(y,t)~dy  - C_1 M(x,t) |x|^{1-\al}-C_2 \|\om\|_{\infty} \quad (x\in D,~i=1,2)
\end{align*}
holds, with universal $C_1, C_2 > 0$.
\end{lem}

\begin{proof}
We prove the result for $Q_2$, the proof for $Q_1$ is similar. We have
\begin{align*}
Q_2(x) \geq	&c_0 \int_{\Om(2x)} \frac{y_1 y_2}{|y|^4}\om(y)~dy + c_0 \int_{\Om(2x)} \left[G_2^2(x, y) -  \frac{y_1 y_2}{|y|^4}\right] \om(y)~dy\\
& + c_0 \int_{\ov\setminus \Om(2x)} G_2^2(x, y) \om(y)~dy - C_1 \|\om\|_{\infty},
\end{align*}
throwing away the nonnegative contribution from $G_2^1$ and estimating $Q_2^r$ by $C_1  \|\om\|_{\infty}$ for $x\in D$.
First, note that straightforward calculations and estimations give
\begin{align*}
\left|G_2^2(x, y)- \frac{y_1 y_2}{|y|^4}\right|= \frac{|y|^4y_2(y_1-x_1)-|y-x|^2|y-\bar x|^2y_1y_2}{|y|^4|y-x|^2|y-\bar x|^2}.
\end{align*}
Using 
\begin{align*}
|y-x|^2&=|y|^2-2x\cdot y +|x|^2\\
|y-\bar x|^2&=|y|^2-2\bar x\cdot y +|x|^2
\end{align*}
the nominator can be estimated by  
\begin{align*}
\sum_{j=1}^4 |x|^j|y|^{6-j}.
\end{align*}
For the denominator, note that
$y \in  \Om(2 x)$ implies that $|y-x|\geq \half|y|, |y-\bar x|\geq \half|y|$, i.e.~the denominator is $\gtrsim |y|^8$.
Hence, the integral over $\Om(2 x)$ is bounded in absolute value by
\begin{align*}
\|\om\|_\infty\sum_{j=1}^4|x|^j \int_{2\ge |y|\ge 2|x|} |y|^{-2-j}dy \ale 1.
\end{align*}
For the estimation of the integral with domain of integration $\ov\setminus \Om(2x)$, we distinguish two cases. 
The more difficult case is given by the condition $x_2 \leq x_1$, and we split the domain of integration into
the three parts $[2 x_1, 1]\times [0, 2x_2], [0, 2 x_1]\times[2 x_1, 1]$ and $[0, 2 x_1]\times [0, 2 x_1]$. 
For the integral over $[2 x_1, 1]\times [0, 2x_2]$, estimate $\om$ by its $L^\infty$-norm and in the remaining integral 
we substitute $y_j = x_j + z_j$. 
\begin{align*}
&\int_{x_1}^{1-x_1} \int_{-x_2}^{x_2} \frac{z_1(x_2+z_2)}{(z_1^2+z_2^2)(z_1^2+(2x_2+z_2)^2)} ~dz
\leq \int_0^1 \int_{-x_2}^{x_2} \frac{2 z_1 x_2}{(z_1^2+z_2^2)(z_1^2+x_2^2)} ~dz_2~dz_1\\
\,\,&\ale \int_0^1 \frac{z_1 x_2}{(z_1^2+x_2^2)} \frac{1}{z_1}\arctan(x_2/z_1)~ d z_1 \ale \arctan(1/x_2) \ale 1.
\end{align*}
 The same strategy for the integral over $[0, 2 x_1]\times[2 x_1, 1]$ leads to
\begin{align*}
\int\limits_{-x_1}^{x_1}\int\limits_{2x_1-x_2}^{1-x_2}  \frac{|z_1|(x_2+z_2)}{(z_1^2+z_2^2)(z_1^2+(2x_2+z_2)^2)}~dz
&\leq 2  \int\limits_{0}^{x_1}\int\limits_{x_1}^{1} \frac{z_1(x_2+z_2)}{(z_1^2+z_2^2)(z_1^2+(2x_2+z_2)^2)}~dz.
\end{align*}
Noting
\begin{align*}
&\int\limits_{0}^{x_1}\int\limits_{x_1}^{1} \frac{z_1 x_2}{(z_1^2+z_2^2)(z_1^2+(2x_2+z_2)^2)}~dz
\leq \int\limits_{0}^{x_1}\int\limits_{x_1}^{1} \frac{z_1 x_2}{(z_1^2+z_2^2)(z_1^2+x_2^2)}~dz_2~ dz_1\\
\,\,&\ale \int\limits_{0}^{x_1} \frac{x_2}{z_1^2+x_2^2}~d z_1 \ale\arctan(x_1/x_2)\ale 1
\end{align*}
and 
\begin{align*}
\int\limits_{0}^{x_1}\int\limits_{x_1}^{1} \frac{z_1 z_2}{(z_1^2+z_2^2)(z_1^2+(2x_2+z_2)^2)}~dz &\leq
\int\limits_{0}^{x_1}\int\limits_{x_1}^{1} \frac{z_1 z_2}{(z_1^2+z_2^2)^2}~dz \ale 1
\end{align*}
we can estimate the integral in question by $C\|\om\|_\infty$.

To estimate the integral over $[0, 2 x_1]\times [0, 2 x_1]$ first note that
\begin{align*}
&\int_{[0, 2 x_1]\times [0, 2 x_1]} G_2^2(x, y)\om(y)~dy\ge\int_{[0,  x_1]\times [0, 2 x_1]} G_2^2(x, y)\om(y)~dy.
\end{align*}
since $\om \ge 0$ and $G_2^2(x, y) \ge 0$ if $y_1 \le x_1$.
We will estimate the integral over $[0,  x_1]\times [0, 2 x_1]$ in absolute value, splitting it again into
$[0, x_1]\times [0, x_1]$ and $[0, x_1]\times [x_1, 2 x_1]$. First, writing $M = M(x,t)$ and using \eqref{omegaIneq} and \eqref{reflectionIneq} we get
\begin{align*}
&\left|\int_{[0, x_1]\times [0, x_1]} G_2^2(x, y)\om(y)~d y \right|\leq \int_{[0, x_1]\times [0, x_1]} \frac{M y_2^{1-\al}}{|y-x|
|y-\bar x|}~d y\\
 \leq &\int_{[0, x_1]\times [0, x_1]} \frac{M |y-\bar x|^{1-\al}}{|y-x|
|y-\bar x|}~d y \le \int_{[0, x_1]\times [0, x_1]} M |y-x|^{-1-\al}~dy\\
\leq &\int_{B(x, r)} M |y-x|^{-1-\al}~dy \leq M r^{1-\al}
\end{align*}
where $B(x, r)$ is the smallest ball around $x$ containing $[0, 2 x_1]\times [0, 2 x_1]$. Clearly $r \ale x_1$,
so the integral is $\ale M x_1^{1-\al}$.

Next, for the remaining part over $[0, x_1]\times [x_1, 2 x_1]$, we estimate $\om$ by $\|\om\|_{\infty}$. We need to bound
\begin{align*}
\int_{[0, x_1]\times [x_1, 2 x_1]} |G_2^2(x, y)|~d y& \leq \int_{-x_1}^0 \int_{x_1-x_2}^{2 x_1-x_2} \frac{|z_1| z_2}{(z_1^2+z_2^2)(z_1^2+(2x_2+z_2)^2)}~dz\\
&~~~+  \int_{-x_1}^0 \int_{x_1-x_2}^{2 x_1-x_2} \frac{|z_1| x_2}{(z_1^2+z_2^2)(z_1^2+(2x_2+z_2)^2)}~ dz.
\end{align*}
For the integral containing $|z_1| z_2$ we distinguish two cases.
In case $x_2 \le \half x_1$, we use $z_1^2+(2x_2+z_2)^2\ge z_1^2+z_2^2$, leading to a bound on the form $\log(1+\frac{x_1}{x_1-x_2})\le C$.
If $x_2 \ge \half x_1$, we use $z_1^2+(2x_2+z_2)^2\ge (x_2+z_2)^2$ in the denominator and $z_2 \le (z_2+x_2)$
in the nominator and get the bound $C x_2^{-1} x_1 \le C$. The integral with $|z_1|x_2$ is estimated as before.

If $x_1\leq x_2$, we split $\ov \setminus \Om(2 x)$ into $[0, 1]\times [0, x_2], [0, 2 x_1]\times[2 x_2, 1]$
and perform similar calculations. In this case, we do not need to use $M(x,t)$. 
\end{proof}

\begin{proof}[Proof of theorem \ref{thmLowerboundQ2}]
Following \cite{KiselevSverak, Zlatos} we observe that the integral 
\begin{align*}
\int_{\Om(2x)} y_1 y_2|y|^{-4}\om(y, t)~dy
\end{align*}
can be bounded
away from zero by an expression of the form $C_1\|\om\|_\infty |\log(1-\mathfrak{m})|$, for $|x|\leq K (1-\mathfrak{m})$.
with universal $C_1, K>0$. Hence we obtain \eqref{ineqthmLowerboundQ2}.
\end{proof}


\subsection{Upper bounds}

The following Lemma gives an upper bound on $Q_1, Q_2$, in terms of $M_{\dD}(t)$. Recall that $d(x)$ is the 
distance to the top of the box, so the upper bound given blows up close to the top of the box. This is, however
not a problem, since we mostly have to integrate $Q_1, Q_2$ along particle trajectories (see the proof 
Theorem \ref{mainTechnicalThm}).
\begin{lem}\label{lemUpperboundQ}
For $x\in D$, 
\begin{align*}
Q_i(x,t) \ale C \|\om\|_{\infty}(1+|\log d(x)|) +  M_{\dD} (t) (|\bde|+\de_3)^{1-\al}\quad (i=1, 2)
\end{align*}
\end{lem}
\begin{proof}
We bound $Q_2$, the calculation for $Q_1$ is analogous. First we note
\begin{align*}
|G_2^k| \ale |y-x|^{-1}|y-\bar x|^{-1} \quad (k=1, 2)
\end{align*}
for $y, x\in \ov$. We write $M=M_{\dD}(t)$, and split the integral in the definition of $Q_2$ into 
into two parts:
\begin{align*}
\int_{\ov} G_2^k (x, y) \om(y)~dy = \int_{\dD}\ldots +   \int_{\ov \setminus \dD}\ldots
\end{align*} 
Since $|\om(y)|\ale M y_2^{1-\al}$ and $y_2 \leq |y-\bar x|$,
\begin{align*}
&\left|\int_{\dD} G_2^k(x, y)\om(y)~dy \right|\ale M \int_{\dD} y_2^{1-\al}  |y-x|^{-1}|y-\bar x|^{-1}~dy \\
&\,\,\ale M \int_{\dD} |y-x|^{-1}|y-\bar x|^{-\al} ~dy \ale M \int_{\dD} |y-x|^{-1-\al} ~dy\\
&\,\,\leq M \int_{B(x, r)} |y-x|^{-1-\al} ~dy \ale M r^{1-\al} 
\end{align*}
where $B(x, r)$ is the smallest ball centered at $x$ containing $\dD$. Obviously $r \ale |\bde|+\de_3$, so the part 
over $\dD$ is dominated by $M (|\bde|+\de_3)^{1-\al}$.

For the part over $\ov \setminus \dD$, we have
\begin{align*}
&\left|\int_{\ov\setminus \dD} G_2^k (x, y) \om(y)~dy\right| \ale \|\om\|_{\infty} \int_{\ov\setminus \dD} |y - x|^{-2} ~dy \\
&\,\,\ale \|\om\|_{\infty} \int\limits_{B(x, 10)\setminus B(x, d(x))} |y - x|^{-2}~ dy 
\ale \|\om\|_{\infty}|\log d(x)|
\end{align*} 
where we have used $|G_2^k| \ale |x-y|^{-1}|y-\bar x|^{-1}$ and $|y - \bar x|\geq |y - x|$ for $x, y\in \ov$. Note also that for $x\in D$,
$\ov \setminus \dD$ is completely contained in $B(x, 10)\setminus B(x, d(x))$ because of \eqref{condDelta}.

For $Q_i^r$ we have the estimate $|Q_i^r(x, t)|\leq C \|\om\|_\infty$, concluding the proof.
\end{proof}
The following important Lemma allows us to control the coefficients of the ODE system \eqref{ode2} in terms of
the quantity $M_{\dD}$. Recall that $d(x)$ is the distance from $x\in D$ to the top of the box.

\begin{lem}\label{lemFor_c_b}
We have the following estimates for $x\in D$:
\begin{align*}
|c(x,t)|& \le C(\al) M_{\dD}(t)x_2^{1-\al} 
+ C(\al, \ga_1, \ga_2) x_2^{1-\ga_1-\ga_2} x_1^{\ga_2} d(x)^{-1+\ga_1},\\
|b(x,t)| &\le C(\al) M_{\dD}(t)x_1^{1-\al}(1+|\log d(x)|) + C(\al, \ga) x_1^{1-\ga} d(x)^{-1+\ga},\\
\left|x_i \frac{\d Q_i}{\d x_i}(x,t) \right|&\le C(\al) M_{\dD}(t)  x_i^{1-\al}(1+|\log d(x)|) + C(\al, \ga) x_i^{1-\ga} d(x)^{-1+\ga}
\end{align*}
where $\ga, \ga_1, \ga_2\in (0, 1), \ga_1+\ga_2 < 1$, $i=1, 2$ and the constants do not depend on $\de_1, \de_2, \de_3,t$.
\end{lem}
\begin{proof}
This is a consequence of Proposition \ref{prop:A4} (see appendix) and the definition of $c, b, x_i \d_{x_i} Q_i$ (see \eqref{def:abc}). Note that we have
\begin{align*}
\left|\frac{\d Q_i^r}{\d x_j}(x,t)\right|\leq C \|\om\|_{\infty} 
\end{align*}
for $x\in D$. When we estimate e.g. $c$, we encounter a term of the form $x_2 \frac{\d Q_2^r}{\d x_1}(x,t)$, obtaining a bound of
the form $C x_2 \|\om\|_{\infty}$, which can be absorbed into 
\begin{align*}
C(\al) M_{\dD}(t)x_2^{1-\al}.
\end{align*} 
\end{proof}

\begin{remark}\label{remAlpha}
It is not possible to set $\al = 0$ in the estimates of Lemma \ref{lemFor_c_b}, i.e.~if we replace $M_{\dD}$ by $\|\nabla \om\|_{D, \infty}$, then
e.g.~the first term on the right-hand side of the estimate for $c$ would contain a logarithmic expression
\begin{align*}
\|\nabla \om\|_{D, \infty} x_2 |\log x_2|.
\end{align*}  
This is the main reason why we do not adopt the stronger feeding condition \eqref{eqfeeding2}, since our main argument cannot be applied to this kind of logarithmic terms.
\end{remark}

\section{Perturbation theory for a system of ordinary differential equations}\label{sec:perturbation}

In this section we derive estimates for an ODE system of the form
\begin{align*}
\dot\bxi(t) = \left(\begin{array}{cc} a(t) & c(t) \\ b(t) & -a(t)\end{array}\right) \bxi(t)
\end{align*}
where $a, b, c$ are given smooth functions on a time interval $[T_0, T_e]$. 
This part is independent of the actual structure of $a, b, c$ from the ODE \eqref{ode2}. 

The idea will be to perturb from the system with $c\equiv 0$, which can be solved
explicitly. We write
\begin{align*}
 P(t):=\left(\begin{array}{cc} a(t) & 0 \\ b(t) & -a(t)\end{array}\right),  ~~S(t):=\left(\begin{array}{cc} 0 & c(t) \\ 0 & 0\end{array}\right).
\end{align*}
\begin{definition}
Let the integral operators $\hP, \hS$ be given by
\begin{align*}
(\hP \bxi)(t) = \int_{T_0}^t P(s)\bxi(s)~ds, ~~(\hS \bxi)(t)=\int_{T_0}^t S(s) \bxi(s) ~ds.
\end{align*}
\end{definition}
Recall that $\displaystyle A(t)=\int_{T_0}^t a(s)~ds$. It is convenient to introduce the following operators: 
\begin{align*}
(F^+ g)(t) &= g(t) + e^{A(t)} \int_{T_0}^t a(s) e^{-A(s)} g(s)~ ds,\\
(F^- g)(t) &= g(t) - e^{-A(t)} \int_{T_0}^t a(s) e^{A(s)} g(s)~ ds.
\end{align*}
\begin{prop}
\begin{enumerate}
\item[(a)] The operator $(I-\hP)$ is bounded and bijective as an operator from $C[T_0, T]$ into $C[T_0, T]$.
\item[(b)] Consider the Volterra integral equation 
\begin{align}\label{ode_eq1}
\phi = \hP \phi + g
\end{align}
with given $g\in C([T_0, T], \R^2)$. The solution  $\phi = (I-\hP)^{-1} g$ is given by
\begin{align}\label{ode_eq3}
\begin{split}
\phi_1(t) &=  F^+ g_1\\
\phi_2(t) &= F^- g_2 + \eAm \int_{T_0}^t \eAp\, b\, F^+ g_1~ds
\end{split}
\end{align}
\end{enumerate}
\end{prop}

\begin{proof}
The statement (a) is standard. Statement
(b) is an easy calculation, noting that \eqref{ode_eq1} is equivalent to the ODE system $\dot \bxi = P\bxi + \dot g$
for $g\in C^1$.
\end{proof}
The initial value problem for the system 
\begin{align*}
\dot \bxi = (P+S)\bxi, \quad \bxi(T_0)~\text{given}
\end{align*}
is equivalent to the Volterra integral equation
\begin{align}\label{ode_perturbedproblem}
\bxi = (\hP+\hS)\bxi + \bxi(T_0).
\end{align}
We can write $\bxi = (I-\hP)^{-1} \bw$ for some $\bw\in C[T_0, T]$. This leads to
\begin{align}\label{ode_eq2}
\bw = \hS (I-\hP)^{-1}\bw + \bxi(T_0).
\end{align}
The following proposition gives a representation of the solution $\bxi$ in terms of $\bw$:
\begin{prop}
Let $\bxi\in C[T_0, T]$ solve the integral equation \eqref{ode_perturbedproblem} with given $\bxi(T_0)$. Then
\begin{align}\label{relWXi}
\begin{split}
\xi_1(t) &= (F^+ w_1)(t), ~~\xi_2(t) = \xi_2(T_0)\eAm + \eAm \int_{T_0}^t \eAp\, b\, F^+ w_1~ ds\\
w_1(t) &= \xi_1(T_0)+\xi_2(T_0) \int_{T_0}^t \eAm \,c~ ds + \int_{T_0}^t \eAm c \int_{T_0}^s \eAp\, b\,F^+ w_1~ d\tau~ds,\\
w_2(t) &= \xi_2(T_0)
\end{split}
\end{align}
\end{prop}

\begin{proof}
First note that 
\begin{align}\label{ode_eq6}
\hS (I-\hP)^{-1} \bw &= \hS \bxi = (\int_{T_0}^t c(s) \xi_2(s)~ ds, 0)
\end{align}
and hence by \eqref{ode_eq2}, $w_2(t) = \xi_2(T_0)$.
It is easy to compute $F^-w_2$:
\begin{align}\label{Fminusw2}
F^-w_2&=F^-\xi_2(T_0)=\xi_2(T_0)e^{-A}.
\end{align}
Recalling $\bxi = (I-\hP)^{-1} \bw$ we get from \eqref{ode_eq1} and \eqref{ode_eq3} with $g=\bw$ and $\phi=\bxi$ and using \eqref{Fminusw2}
\begin{align}\label{eqxi1xi2}
\left(\begin{array}{c} \xi_1\\ \xi_2 \end{array}\right)&= \bxi=(I-\hP)^{-1} \bw=\left(\begin{array}{l}F^+ w_1 \\ \xi_2(T_0)e^{-A} + \eAm \int_{T_0}^t \eAp\, b\, F^+ w_1~ds \end{array}\right)
\end{align}
which is the first line of \eqref{relWXi}.
From \eqref{ode_eq6} follows
\begin{align*}
\left(\begin{array}{c} w_1\\ w_2 \end{array}\right)&= \hS\bxi+\bxi(T_0)=\left(\begin{array}{c}\int_{T_0}^tc\,\xi_2~ds+\xi_1(T_0) \\ \xi_2(T_0)\end{array}\right).
\end{align*}
Together with \eqref{eqxi1xi2} we get
\begin{align*}
w_1&=\int_{T_0}^tc\,\xi_2~ds+\xi_1(T_0)=\int_{T_0}^tc\,\left(\xi_2(T_0)e^{-A} + \eAm \int_{T_0}^s \eAp\, b\, F^+ w_1~d\tau\right)~ds+\xi_1(T_0)\\
&= \xi_1(T_0)+\xi_2(T_0) \int_{T_0}^t \eAm \,c~ ds + \int_{T_0}^t \eAm c \int_{T_0}^s \eAp\, b\,F^+ w_1~ d\tau~ds
\end{align*}
\end{proof}
We will need the following Gronwall-type inequality by Wilett \cite{Wilett, Walter}:
\begin{lem}\label{lem_Willet}
Let $z, f_0, f_1, f_2, v_1, v_2$ are nonnegative, integrable functions on $[T_0, T]$ and suppose $z$ satisfies the following
integral inequality:
\begin{align*}
z(t) \leq f_0(t) + f_1(t) \int_{T_0}^t v_1(s) z(s)~ds + f_2(t) \int_{T_0}^t v_2(s) z(s)~ds. 
\end{align*}
Then $z \leq H f_0 $, where $H$ is the following functional 
\begin{align}\label{eq_H}
\begin{split}
(H f_0)(t) = & f_0 + f_1 \expo{\intto v_1 f_1}\intto v_1 f_0\\
&+\left[f_2(t)+f_1(t)\expo{\intto v_1 f_1}\intto{v_1 f_2}\right]\\
&\quad\times \expo{\intto v_2\left[f_2(s)+f_1(s)\expo{\intso v_1 f_1}\intso{v_1 f_2}\right]}\\
&\qquad\times \intto v_2 \left[f_0(s)+f_1(s)\expo{\intso v_1 f_1}\intso{v_1 f_0}\right]
\end{split}
\end{align}
We write $H f_0$ to emphasize the linear dependency on $f_0$.
\end{lem}
\begin{proof}
We give the proof for reference. Recall first the following basic form of Gronwall's integral inequality \cite{Walter}. Suppose
$z, r, f_1, v_1$ are nonnegative functions on $[T_0, T]$ satisfying the integral inequality
\begin{align*}
z(t) \leq r(t) + f_1(t)\intto v_1 z~ds,
\end{align*}
then
\begin{align}\label{ode_eq5}
z(t) \leq r(t)+f_1(t) \expo{\intto v_1 f_1}\intto v_1 r~ds~~\quad (t\in [T_0, T]).
\end{align}
Set $r = f_0 + f_2 \intto v_2 z$ and apply \eqref{ode_eq5}. This leads to the following bound for $z$:
\begin{align}\label{ode_ineq3}
z(t) \leq f_0 + f_2 \intto v_2 z +f_1(t) \expo{\intto v_1 f_1}\intto v_1 \left[f_0 + f_2 \intso v_2 z\right].
\end{align}
Note that 
\begin{align*}
\intto v_1 f_2 \intso v_2 z & \leq \left(\intto v_1 f_2\right)\intto v_2 z
\end{align*}
since $v_1, f_2, z, v_2 \geq 0$. Thus \eqref{ode_ineq3} implies
\begin{align*}
z(t) &\leq f_0 + f_1 \expo{\intto v_1 f_1} \intto v_1 f_0 + \\
& +\left[f_2(t)+ f_1(t) \expo{\intto v_1 f_1}\left(\intto v_1 f_2\right)\right] \intto v_2 z. 
\end{align*}
Applying \eqref{ode_eq5} again, this time with $r = f_0 + f_1 \expo{\intto v_1 f_1}\intto v_1 f_0$, yields
the result \eqref{eq_H}.
\end{proof}
\begin{lem}
Let $\bxi\in C[T_0, T]$ solve the integral equation \eqref{ode_perturbedproblem} with given $\bxi(T_0)$. Then
the estimates 
\begin{align}\label{ode_ineq5}
\begin{split}
|\xi_1(t)| &\leq (H f_0)(t)+ \eAp \intto v_1 H f_0\\
|\xi_2(t)| &\leq |\eAm \xi_2(T_0)| + \eAm \left[ \intto v_2 H f_0 + \intto e^{2 A} |b| \intso v_1 H f_0\right],
\end{split}
\end{align}
hold, where $H$ is the functional \eqref{eq_H} and where
\begin{align*}
&f_1(t) = \intto \eAm |c| \intso e^{2 A} |b|,
&f_2(t)= \intto \eAm |c|,\\
&f_0(t) = |\xi_1(T_0)|+f_2(t) |\xi_2(T_0)|,
&v_1(t) = |a(t)| \eAm,\\
&v_2(t) = |b(t)| \eAp.
\end{align*}
\end{lem}

\begin{proof}
Using obvious estimations, we get from \eqref{relWXi} the following integral inequality for $|w_1|$:
\begin{align*}
|w_1(t)|&\leq |\xi_1(T_0)| + |\xi_2(T_0)| \intto \eAm |c|+\intto \eAm |c| ds ~ \intto \eAp |b| |w_1| ds\\
&+  \intto \eAm |c| \intso e^{2 A} |b| d\tau ds~ \intto |a|\eAm |w_1|\\
&= f_0(t) + f_1(t) \intto v_1 |w_1|ds + f_2(t) \intto v_2 |w_1|ds, 
\end{align*}
where the expressions $f_0, f_1, f_2, v_1, v_2$ are given as in the statement of the Lemma.
Now using Lemma \ref{lem_Willet}, we obtain $|w_1(t)|\leq H f_0$ on $[T_0, T]$. The inequalities 
\eqref{ode_ineq5} follow from the formulas \eqref{ode_eq3}.
\end{proof}

\begin{remark}
The reader might wonder why we did not perturb from a diagonal system, i.e. regard also
$b$ as a perturbation like $c$ as in the heuristic discussion. 
While it is certainly possible to derive corresponding perturbation formulas for $\xi_1, \xi_2$,
it turns out that the balance of growing and decaying factors is not favorable for the
arguments in Section \ref{sec:main}. Fortunately, the perturbation from the tridiagonal system 
behaves in a more stable way.  
\end{remark}

\section{Main argument}\label{sec:main}

\subsection{The main technical result}
In order to formulate our main technical result, we introduce a notion of {\em harmless nonlinear bound}.

\begin{definition}
A function  $\cN = \cN(R, \beta, \al, \bde, M)$ where all arguments are nonnegative numbers
 is a {\em harmless nonlinear function} if for fixed $\al\in (0, 1), \beta>0$ the following holds: 
For any given $R > 0$, there exists  $\bar \de_2(R) > 0$ and a number $\bar \de_1=\bar \de_1(R, \bar \de_2)>0$
such that for all $\de_2 \leq \bar \de_2, \de_1 \leq \bar \de_1$ the inequality
\begin{align*}
\cN(R, \beta, \al, \bde, R) < R
\end{align*} 
holds.
\end{definition}

Recall the box $\dD$ is said to satisfy the conditions of {\em controlled feeding} if there is a $R \geq 0$ with
\begin{align*}
|\d_{x_1}\om(x, t)|\leq R x_2^{1-\al}, ~~|\d_{x_2}\om(x, t)|\leq R \quad (x\in \dD\setminus D)
\end{align*} 
for all times $t \geq 0$. $R$ is called \emph{feeding parameter}. For convenience, we introduce the following definition.
\begin{definition}
Let $T > 0, \beta > 0$. We say that the flow is $\beta$-hyperbolic in the box $D$ on $[0, T]$ if 
\begin{align*}
Q_i(x, t) \geq \beta \quad (x\in D, \,t\in [0, T], i=1, 2).
\end{align*}
\end{definition}
\begin{thm}\label{mainTechnicalThm}
Let $0 < \al < 1/4$. There exists a harmless nonlinear function $\cN = \cN(R, \beta, \al, \bde, M)$ (determined by
a-priori known data) with the following properties.
If $\om$ is a solution of the Euler equation, $\dD$ a box defined by \eqref{defD} with parameters $\de_1, \de_2, \de_3 > 0$
satisfying \eqref{condDelta} and $T > 0$ is such that 
\begin{enumerate}
\item[(i)] the flow is $\beta$-hyperbolic in the box $D$ on the time interval $[0, T]$, 
\item[(ii)] the box $\dD$ is satisfies the conditions of {\em controlled feeding} with parameter $R>\|\om\|_\infty$,
\item[(iii)] the initial data satisfies,
\begin{align*}
M_{D}(0) < R,~~\left|\frac{\d \om_0}{\d x_1}(x)\right| \le R x_2^{1-\al},~~\left|\frac{\d \om_0}{\d x_2}(x)\right|\leq R\qquad (x \in D),
\end{align*}
\item[(iv)] there exists a number $K$ such that
\begin{align*}
M_{D}(t) \leq K  \quad (t\in [0, T]), 
\end{align*}
\end{enumerate}
then
\begin{align*}
M_{D}(t) \leq \cN(R, \al, \beta, \bde, K) \quad (t\in [0, T])
\end{align*}
holds.
\end{thm}


\subsection{Estimates along particle trajectories}
In this section we develop the technical tools to prove Theorem \ref{mainTechnicalThm}. The proofs for the estimates for $f_0,f_1,f_2$ and $Hf_0$ along trajectories are 
heavily interconected (see Figure \ref{figmap}). We advise the reader to concentrate on the main flow of arguments indicated by the bold arrows and boxes in the map of Section \ref{sec:main}.

\begin{figure}[htbp]
\begin{center}
\includegraphics[scale=0.55]{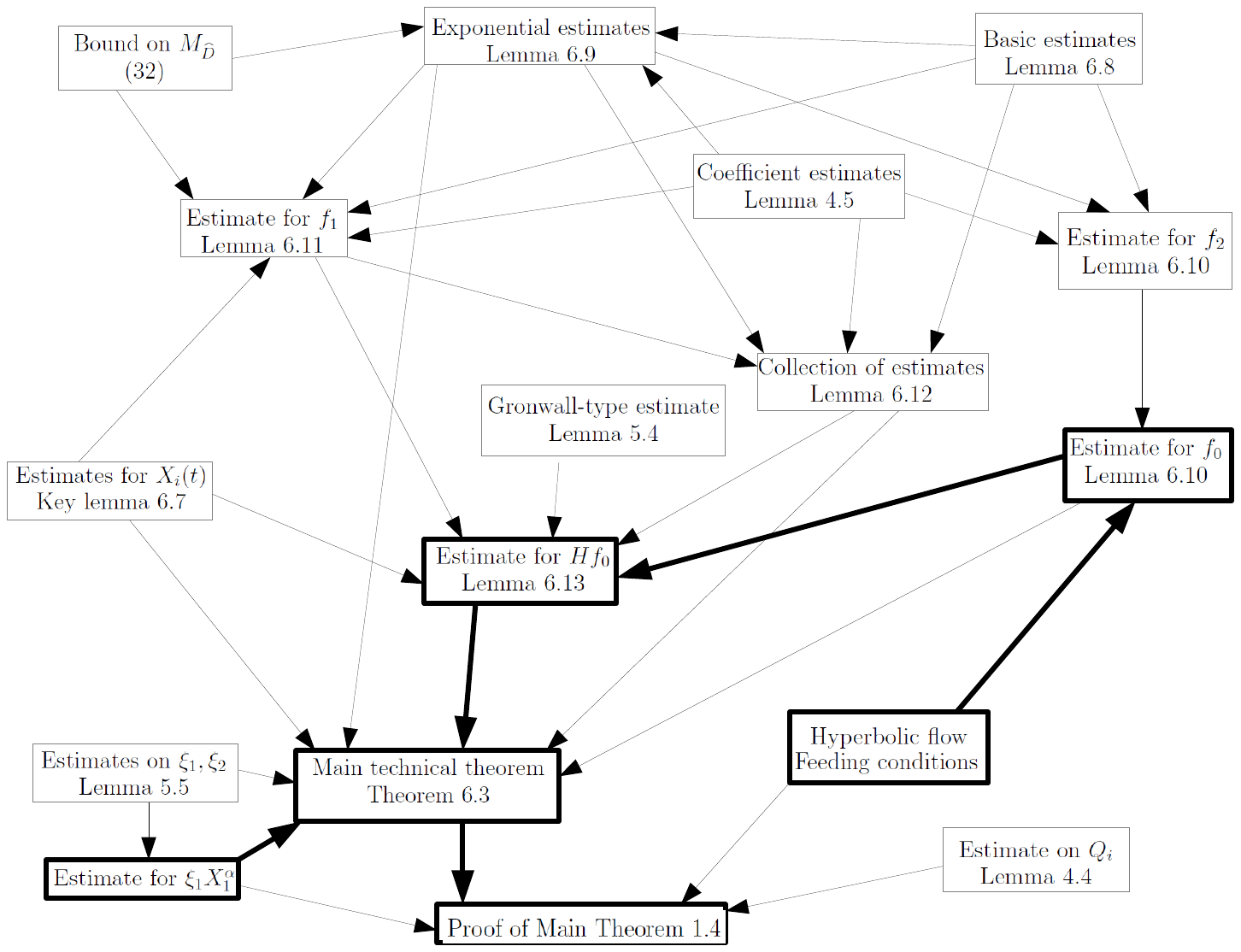}%
\end{center}
\caption{Map of Section \ref{sec:main}.} \label{figmap}
\end{figure} 
Let $\om$ be a given double-odd solution 
of the Euler equation that is in $C^1([0, \infty), C^2(\To))$. 
Moreover, let $\dD$ be a box depending on the parameters $\de_1, \de_2, \de_3 > 0$ satisfying the conditions \eqref{condDelta}.

Suppose also that for the remainder of this section, (i)-(iv) from theorem \ref{mainTechnicalThm} are satisfied.
For abbreviation, we write in the following
\begin{align*}
M := \max\{K, R\}.
\end{align*}
We observe the following important fact: since $\de_1, \de_2, \de_3 \leq 1$,
\begin{align}\label{ineq_M}
M_{\dD}(t) \leq M
\end{align}
holds.

We consider associated particle trajectories, which are the solutions of 
\begin{align}\label{partODE}
\dot X_1 = - X_1 Q_1, \,\,\,\dot X_2 = X_2 Q_2.
\end{align}
More precisely, we define the particle trajectories as follows: for any $(x_0, t_0)\in \ba D\times [0, \infty)$ we take the maximal
solution of $t\mapsto \bX(t)$ of \eqref{partODE} which passes through $(x_0, t_0)$, and lies $\ba D$. 
$\bX$ is defined on an interval $[T_0, T_e]$ such that 
\begin{enumerate}
\item[(i)] $\bX(t)\in \ba D$ for all $T_0 \le t \le T_e$,
\item[(ii)] either $T_0 = 0$ or $T_0 > 0$, in which case necessarily $\bX(T_0)\in \d D$,
\item[(iii)] $\bX(T_e) \in \d D$.
\end{enumerate}
Observe that $\bX$ is given by
\begin{align}\label{X2}
\begin{split}
X_1(t) &= X_1(T_0) \exp\left(- \intto Q_1(\bX(s), s)~ds\right) \\
X_2(t) &= X_2(T_0) \exp\left(\intto Q_2(\bX(s), s)~ds\right).
\end{split}
\end{align}
We call $T_0$ the entry time and $T_e$ the exit time of a particle trajectory. 
$T_0=0$ if the particle starts in $\ba D$ for $t=0$.

The next proposition gives a upper bound for the time a particle can spend in the upper
half of the box $D$, provided the flow is $\beta$-hyperbolic.

\begin{prop}\label{propTime}
Suppose that the flow is $\beta$-hyperbolic in the box $D$ on the time interval $[0, T]$. Let $\bX$ be a particle trajectory
whose entry time $T_0$ is smaller than $T$.Then if $X_2(T_0)\neq 0$ there is  either a time $T_1$, $T_e > T_1\geq T_0$ such that
\begin{align*}
X_2(t) \geq \half \de_2 \quad (t\in [T_1, T])
\end{align*}
or 
\begin{align*}
X_2(t) \leq \half \de_2 \quad (t\in [T_0, T]).
\end{align*}
If $T_1$ exists, we have the estimate
\begin{align*}
T_e - T_1 \leq {\beta}^{-1} \log(2).  
\end{align*}
\end{prop}

\begin{proof} The statement on the time $T_1$ follows directly form the fact that the flow is $\beta$-hyperbolic in the box. If $T_1$ exists, we have analogously to \eqref{X2}
\begin{align*}
\de_2&=X_2(T_e)=X_2(T_1)\exp\left(\int_{T_1}^{T_e}Q_2~ds\right)\ge \frac{\de_2}{2}\exp\left(\int_{T_1}^{T_e}Q_2~ds\right)\ge\frac{\de_2}{2}\exp\left(\beta(T_e-T_1)\right).
\end{align*}
Solving for $T_e-T_1$ gives the result.
\end{proof}

\begin{definition}
We call a function $g=g(\al, \beta, \bde, M)$ \emph{harmless generic factor} it has the following property:
There exists a $p_0 > 0$ such that for all $p\ge p_0$ and fixed $\al, \beta, M$ 
\begin{align*}
g(\al, \beta, \de_2^p, \de_2, M)
\end{align*}
is bounded as $\de_2\to 0$.
\end{definition}

\begin{remark}
For example, a function of the form
\begin{align*}
g=C(\al, \beta) \left[\de_2^{\ga_3} M(1+|\log \de_2|)+\de_1^{\ga_1}{\de_2}^{-\ga_2}+1\right]^{\ga_4}+C(\al,\beta,\ga_j)
\end{align*}
($\gamma_j > 0$) is a harmless generic factor, and $e^g$ is also a harmless generic factor if $g$ is one.
When performing estimations, we shall often absorb harmless generic factors into one
another, so the actual meaning of $g$ may change from line to line.

In our argument there will appear only finitely many different generic factors (although all denoted by $g$). To make the boundedness of them all work as $\de_2\to 0$ we just pick a $p$ that is bigger than all the $p_0$ of all appearing generic factors.  
\end{remark}

Our goal will be to obtain estimates for the quantities $f_0, f_1, f_2, v_1, v_2$ along a single particle trajectory,
up to the given time $T$, so that we can apply our ODE estimates from section \ref{sec:perturbation}.
The crucial point is that our bounds depend \emph{not directly} on $\om,~T,~T_e$ but 
only on $\beta, \alpha, \bX(T_0)$.  
For the estimations below we often refer to a fixed particle trajectory with entry time $T_0$, along
which we evaluate integrals over time of the quantities $Q_1, Q_2, c$ etc. To make the notation more compact, 
we often skip $\bX$ in the arguments of the integrands, e.g. we write
\begin{align*}
\intto |c| e^{-A} ~ds=\intto |c| e^{-A(s)} ds = \intto |c(\bX(s), s)| \exp\left(\intso a(\bX(\tau), \tau )~ d\tau\right) ~ds.
\end{align*}


\begin{lem}\label{keyLemma}
For any $t\le T_e$,
\begin{align*}
X_2(T_0) \le \de_2 \expo{-\int_{T_0}^{t} Q_2(\bX(s),s)~ ds}.
\end{align*}
\end{lem}

\begin{proof}
Since the particle trajectory lies in $D$ for $t\in [T_0, T_e]$, 
\begin{align*}
\de_2 \ge X_2(t) = X_2(T_0) \expo{\int_{T_0}^{t} Q_2~ ds}
\end{align*}
holds.
\end{proof}


Let $\phi: [0, \infty) \to [0, \infty)$ be a function with the properties
\begin{align*}
\phi(s) \leq 1-e^{-s}
\end{align*}
and $\phi$ monotone nondecreasing on $[0, \infty)$, 
$\phi$ linear on $[0, s^*]$ and $\phi$ constant on $[s^*, \infty)$ for some $s^*$. We fix such a function $\phi$ for the following.
\begin{prop}\label{propSomeEst}
Along a particle trajectory in a $\beta$-hyperbolic flow in $D$, we have the following
for $t\in [T_0, \min \{T_e,T\}]$:
\begin{enumerate}
\item[(i)]\begin{align*}
X_1(t) &\leq \de_1 \expo{-\beta (t-T_0)}\\
X_2(t) &\leq \de_2 \expo{-\beta (\min \{T_e,T\}-t)},
\end{align*}
\item[(ii)]\begin{align*}
d(\bX(t)) \geq \de_2 \phi\left(\int_{t}^{\min \{T_e,T\}} Q_2 ~ds\right) \geq \de_2 \phi(\beta(\minT-t)),
\end{align*}
\item[(iii)] Suppose $T_1$ from proposition \ref{propTime} exists. Then the following holds for any $\ga\in (0,1)$ and $t \in [T_1, \min\{T_e, T\}]$:
\begin{align*}
\int_{T_1}^{t} d(\bX(s))^{-1+\ga}~ds \leq C(\ga, \beta) \de_2^{-1+\ga},\\
\int_{T_1}^{t} |\log d(\bX(s))|~ds \leq C(\beta) |\log\de_2|
\end{align*}
with $C(\beta),C(\ga, \beta)$ independent of the trajectory.
\end{enumerate}
\end{prop}

\begin{proof}
For (i), recall that under the assumption of $\beta$-hyperbolic flow, 
$Q_2 \geq \beta$. From \eqref{X2}, we get
\begin{align}
&X_2(t) = X_2(T_0)  \expo{\int_{T_0}^{\min \{T_e,T\}} Q_2~ds
- \int_{t}^{\min \{T_e,T\}} Q_2~ds}\nn\\ 
&\,\,\,= X_2(\min \{T_e,T\})\expo{- \int_{t}^{\min \{T_e,T\}} Q_2~ds}\nn \\
 &\,\,\,\leq \de_2 \expo{-\beta (\min \{T_e,T\}-t)}\label{propSomeEstEq1},
\end{align}
noting that $X_2(\min \{T_e,T\})\leq \de_2$. The bound for $X_1$ is analogous.

Now we show (ii). Recall that $d(\bX) = \de_2 - X_2(t)$. Hence by \eqref{propSomeEstEq1}
\begin{align*}
\de_2 - X_2(t)& = \de_2\left(1- \expo{-\int_t^{\min \{T_e,T\}} Q_2 ~ds}\right) \ge \de_2 \phi\left(\int_t^{\min \{T_e,T\}} Q_2 ~ds\right)\\
&\ge\de_2 \phi\left(\beta(\min \{T_e,T\}-t)\right).
\end{align*}

(iii) We split the integrals by introducing the time 
$T^*$ defined as follows: $T^*$ is the maximum of all $s\in [T_1,\min \{T_e,T\}]$ such that 
\begin{align*}
&\phi(\beta(\min\{T_e, T\}-s))= \phi(s^*).
\end{align*} 
If there are no such $s$, we set $T^*= T_1$. Thus we split the integrals in (iii)
as follows:
\begin{align*}
\int_{T_1}^{t}= \int_{T_1}^{T^*} \ldots +  \int_{T^*}^{t}\ldots 
\end{align*}
if $t \geq T^*$, otherwise we have only one integral from $T_1$ to $t$.
We calculcate 
\begin{align*}
\int_{T_1}^{T^*} d(\bX(s))^{-1+\ga}~ds &\leq  
\de_2^{-1+\ga} \int_{T_1}^{T^*} \phi(\beta(\min\{T_e, T\}-s))^{-1+\ga}~ds\\
&\leq \de_2^{-1+\ga} (T_e-T_1) \phi(s^*)^{-1+\ga} \leq C(\beta, \ga)\de_2^{-1+\ga}\\
\int_{T^*}^{t} d(\bX(s))^{-1+\ga}~ds &\leq 
\de_2^{-1+\ga} \int_{T^*}^{t} \phi(\beta(\min\{T_e, T\}-s))^{-1+\ga}~ds\\
&\ale \de_2^{-1+\ga} \beta^{-1+\ga}\int_{T^*}^{t} (\min\{T_e, T\}-s)^{-1+\ga}~ds\\
&\ale \de_2^{-1+\ga} \beta^{-1+\ga}\int_{T_1}^{\min\{T_e, T\}} (\min\{T_e, T\}-s)^{-1+\ga}~ds\\
&\ale \de_2^{-1+\ga} \beta^{-1+\ga}\int_{0}^{T_e-T_1} z^{-1+\ga}~dz
\ale \de_2^{-1+\ga} C(\beta, \ga).
\end{align*}
using (ii), Proposition \ref{propTime} to estimate $T_e-T_1$ and the fact that $\phi$ is linear on $[0, s^*]$.
 The second integral is treated analogously.
\end{proof}

\begin{lem}\label{lemExponential}
Along a particle trajectory, we have, for $T_0\leq t\leq \minT$,
\begin{align*}
e^{\pm A(t)} &\leq g(\al, \beta, \bde, M) \expo{\pm \intto Q_2(s)~ds}, \\
\exp\left(\pm \intto Q_i(s) ~ds\right) &\leq g(\al, \beta, \bde, M) \exp\left(\pm \intto Q_j(s)~ ds\right), \quad i, j=1,2
\end{align*}
where $g(\al, \beta, \bde, M)$ are harmless generic factors depending only on the quantities indicated.
\end{lem}

\begin{proof}
We prove the first inequality of the Lemma, for the other we use similar arguments.
Recall $a(t) = Q_2(t) + X_2(t) \d_{x_2} Q_2(t)$, $A(t)=\int_{T_0}^t a(s)~ds$ and thus
\begin{align*}
&\pm A(t) \leq \pm \intto Q_2(s) ~ds  + \intto |X_2(s) \d_{x_2} Q_2(s)|~ds.
\end{align*}
We now use Lemma \ref{lemFor_c_b} and \eqref{ineq_M}:
\begin{align*}
\intto |X_2(s) \d_{x_2} Q_2(s)|~ds \le  C(\al) M \inttoet X_2^{1-\al}(1+|\log d(\bX)|) ~ ds\\
  + C(\al, \ga) \inttoet  X_2^{1-\ga}d(\bX)^{-1+\ga}~ds.
\end{align*}
Note that the interval of integration has been enlarged. 
With $T_1$ from proposition \ref{propTime} we split the interval of integration into $[T_0, T_1]$ and $[T_1,  \min\{T, T_e\}]$ provided $\min\{T, T_e\} \geq T_1$. 
The case $\min\{T, T_e\} < T_1$ is analogous.

In the part over $[T_0, T_1]$, while $d(\bX)\geq \half \de_2$, we cannot control the length of the time interval, so we estimate as follows:
\begin{align*}
\int_{T_0}^{T_1}X_2^{1-\al}(1+|\log d(\bX)|) &\leq  \de_2^{1-\al}\int_{T_0}^{T_1} e^{-(1-\al)\beta (\minT-s)} (C+|\log \de_2|)~ds\\
&\le  C \de_2^{1-\al}  |\log \de_2| \int_{0}^{\infty} e^{-(1-\al)\beta z}~dz\\
&\le C(\al, \beta) \de_2^{1-\al}|\log \de_2|,
\end{align*}
using part (i) of proposition \ref{propSomeEst} and $d(\bX(s))\geq \half \de_2$ for $s\in [T_0, T_1]$, and $\de_2$ sufficiently small.
In the part over $[T_1, \minT]$ the length of the time interval is bounded but
$|\log d(\bX)|$ is unbounded, so we proceed differently:
\begin{align*}
\int_{T_1}^{\min\{T, T_e\}} X_2^{1-\al}(1+|\log d(\bX)|) &\le \de_2^{1-\al} \int_{T_1}^{\min\{T, T_e\}}|\log d(\bX)|~ds\\
&\le C(\beta) \de_2^{1-\al}|\log\de_2|.
\end{align*}
using statement (iii) of Proposition \ref{propSomeEst} and $X_2\leq \de_2$.

For the second integral involving $X_2^{1-\ga} d(\bX)^{-1+\ga}$, we note 
\begin{align*}
\int_{T_0}^{T_1} X_2^{1-\ga} d(\bX)^{-1+\ga_1}&\le C(\ga)  \de_2^{-1+\ga} \int_{T_0}^{T_1} 
\left(\de_2 e^{-\beta(\minT-s)}\right)^{1-\ga} ds\\
&\le  C(\ga,\beta) \\
\int_{T_1}^{\minT} X_2^{1-\ga}d(\bX)^{-1+\ga} &\le \de_2^{1-\ga}\int_{T_1}^{\minT}d(\bX)^{-1+\ga}\\
&\le C(\ga,\beta)
\end{align*}
by proposition \ref{propSomeEst}, (i) and (iii) and moreover using $X_2\leq \de_2$.
This yields finally
\begin{align*}
\intto |X_2(s) \d_{x_2} Q_2(s)|~ds \leq [C(\al, \beta)M \de_2^{1-\al}|\log \de_2|+C(\ga,\beta)] 
\end{align*}
implying the result, since the term in square brackets is a harmless generic factor. 

To prove the second inequality, we use (the velocity field is divergence-free)  
$$
Q_1(t) + X_1(t) \d_{x_1} Q_1(t) =  Q_2(t) +X_2(t) \d_{x_2} Q_2(t)
$$
implying $|Q_i| \le |Q_j| + \sum_{k=1}^2|x_k\partial_{x_k} Q_k|$. The expressions involving
$x_k\partial_{x_k} Q_k$ are estimated as before.
\end{proof}

\subsection{Estimates for $f_0,f_1,f_2,v_1,v_2$ and $Hf_0$.}
\begin{lem}\label{lemf2f0}
The following estimates hold for $T_0\leq t \leq \minT$:
\begin{align}
f_2(t)  &\leq  g(\al, \beta, \bde, M)  X_2(T_0)^{1-\al}, \label{lemf2f0_est1}\\
f_0(t) &\leq R\, g(\al, \beta, \bde, M)  X_2(T_0)^{1-\al}\nn.
\end{align}
\end{lem}
\begin{proof} We write $g = g(\al, \beta, \bde, M)$ for any occuring harmless
factor. Using Lemma \ref{lemFor_c_b} with $\ga_1=\ga_2=\frac{\al}{2}$, 
\begin{align*}
f_2(t) &= \intto e^{-A}|c| \ale M \inttoet e^{-A} X_2^{1-\al}~ ds\\
&\,\,+ C(\al)\inttoet e^{-A} X_2^{1-\al}X_1^{\al/2} d(\bX)^{-1+\al/2} ds.
\end{align*}
Observe first that by Lemma \ref{lemExponential} $e^{-A(s)}$ is estimated by $\expo{- \int_{T_0}^s Q_2~ d\tau}$ and thus using $Q_2\geq \beta$ again we get
\begin{align}\label{lem_eq1}
e^{-A} X_2(s)^{1-\al} 
&\leq g X_2(T_0)^{1-\al} \expo{-\al\int_{T_0}^s Q_2 ~d\tau}\leq g X_2(T_0)^{1-\al} \expo{-\al \beta (s - T_0)}.
\end{align}
Employing \eqref{lem_eq1} to estimate the integral containing $e^{A}X_2^{1-\al}$ yields: 
\begin{align*}
\int_{T_0}^{\minT}\!\!\!\!\! e^{-A} X_2^{1-\al} ds
&\leq g\,X_2(T_0)^{1-\al}  \int_{T_0}^{\infty} e^{-\al \beta(s-T_0)} ds\leq g\,X_2(T_0)^{1-\al} C(\al, \beta).
\end{align*}
For the integral containing $e^{-A} X_2^{1-\al}X_1^{\al/2} d(\bX)^{1-\al/2}$, we use \eqref{lem_eq1} again and estimate
\begin{align*}
&\inttoet\!\!\!\!\! e^{-A} X_2^{1-\al}X_1^{\al/2} d(\bX)^{1-\al/2} ds \le 
g X_2(T_0)^{1-\al} \de_1^{\al/2} \inttoet\!\!\!\!\! e^{-\al \beta(s-T_0)}d(\bX)^{-1+\al/2}ds.
\end{align*}
As in the proof of Lemma \ref{lemExponential}, we split the interval of integration into $[T_0, T_1]$ and $[T_1, \minT]$ in case $T_1 \leq \minT$,
obtaining
\begin{align}
\label{expintest1}\int_{T_0}^{T_1} e^{-\al \beta(s-T_0)}d(\bX)^{-1+\al/2}~ds  &\ale \de_2^{-1+\al/2},\\
\label{expintest2}\int_{T_1}^{\minT} e^{-\al \beta(s-T_0)}d(\bX)^{-1+\al/2}~ds &\ale \int_{T_1}^{\minT} d(\bX)^{-1+\al/2}~ds\ale \de_2^{-1+\al/2}
\end{align}
where we have used $d(\bX)\ge \half \de_2$ for \eqref{expintest1} and $e^{-\al \beta(s-T_0)}\leq 1$ and 
Proposition \ref{propSomeEst} for \eqref{expintest2}. The case $T_1 \geq \minT$ is covered by \eqref{expintest1}.
To estimate $f_0$, we use that the feeding condition holds and that assumption (iii) from Theorem \ref{mainTechnicalThm} holds. This gives 
\begin{align*}
|\xi_1(T_0)| &= |\d_{x_1}\om(\bX(T_0), T_0)| \leq R X_2(T_0)^{1-\al},\\
|\xi_2(T_0)|&=|\d_{x_2}\om(\bX(T_0), T_0)|\leq R
\end{align*}
for both of the cases $T_0=0$ (particle starts in $D$) and $T_0>0$ (particle starts in or crosses the feeding zone before entering $D$). 
Now use the definition of $f_0$ and the estimate \eqref{lemf2f0_est1} for $f_2$.
\end{proof}
\begin{lem}\label{lemf1}
For $T_0\leq t\leq \minT$, 
\begin{align*}
f_1(t) &\leq g(\al, \beta, \bde, M) \de_1^{1-\al} \de_2^{\al} e^{ \al \intto Q_2 \,ds} 
\end{align*}
with a harmless generic factor $g$ depending on the quantities indicated. 
\end{lem}
\begin{proof}
We abbreviate again $g = g(\al, \beta, \bde, M)$. First we claim that for sufficiently small $\de_2$
\begin{align}\label{lemf1_claim}
\intto e^{2 A} |b|~ds &\le g\, X_1(T_0)^{1-\al}\left[M|\log \de_2|+\de_1^{\frac{\al}{2}}{\de_2}^{-1+\frac{\al}{2}}\right]e^{(1+\al)\intto Q_2\,ds}.
\end{align}
We treat the case $T_1 \leq t \leq \minT$. Using Lemma \ref{lemFor_c_b}
 with $\ga= \frac{\al}{2}$,  and Lemma \ref{lemExponential} we get
\begin{align*}
e^{2 A} |b| &\leq e^{2 A} X_1^{1-\al}[M_{\dD}(t)(1+|\log d(\bX)|)+X_1^{\frac{\al}{2}} d(\bX)^{-1+\frac{\al}{2}}]\\
 &\leq g \,e^{(1+\al)\int_{T_0}^s Q_1 ds} \,X_1(T_0)^{1-\al}[M(1+|\log d(\bX)|)+\de_1^{\frac{\al}{2}} d(\bX)^{-1+\frac{\al}{2}}].
\end{align*}
Also recall that $M_{\dD}(t)\leq M$.
To integrate this bound from $T_0$ to $t$ we split into two integrals from $T_0$ to $T_1$ and $T_1$ to $t$. For $t\in [T_0,T_1]$ we can estimate the factor in square brackets independent of $t$ using $d(\bX)\ge\half \de_2$:
\begin{align*}
g\,X_1(T_0)^{1-\al}[M(1+|\log \de_2|)+\de_1^{\frac{\al}{2}} \de_2^{-1+\frac{\al}{2}}]\int_{T_0}^{T_1} e^{(1+\al)\int_{T_0}^{s} Q_1 d\tau}~ds.
\end{align*}
The remaining integral can be estimated as follows:
\begin{align*}
\int_{T_0}^{T_1}  e^{(1+\al)\int_{T_0}^{s} Q_1 d\tau}~ds &\le \int_{T_0}^{t}  \frac{Q_1}{Q_1}  e^{(1+\al)\int_{T_0}^{s} Q_1 d\tau}~ds
\leq \beta^{-1} (1+\al)^{-1}  \left.e^{(1+\al)\int_{T_0}^{s} Q_1 d\tau}\right|_{s=T_0}^{s=t}\\
 &\ale e^{(1+\al)\int_{T_0}^{t} Q_1 d\tau}.
\end{align*}
Hence for sufficiently small $\de_2$
\begin{align*}
 \int_{T_0}^{T_1} e^{2 A} |b|~ds  &\ale g\, X_1(T_0)^{1-\al}[M|\log \de_2|+\de_1^{\frac{\al}{2}} \de_2^{-1+\frac{\al}{2}}]e^{(1+\al)\int_{T_0}^{T_1} Q_1}.
\end{align*}
For the remaining part $\int_{T_1}^{t} e^{2 A} |b|~ds$, we use Proposition \ref{propSomeEst} again, and find the bound for small $\de_2$
\begin{align*}
&g\,X_1(T_0)^{1-\al} 
\int_{T_1}^{t}e^{(1+\al)\int_{T_0}^s Q_1 ds} [M(1+|\log d(\bX)|)+\de_1^{\frac{\al}{2}} d(\bX)^{-1+\frac{\al}{2}}]~ds\\
&\le g\,e^{(1+\al)\int_{T_0}^t Q_1 ds} X_1(T_0)^{1-\al} 
\int_{T_1}^{t}[M(1+|\log d(\bX)|)+\de_1^{\frac{\al}{2}} d(\bX)^{-1+\frac{\al}{2}}]~ds\\
&\,\,\ale g\, X_1(T_0)^{1-\al} 
[M|\log \de_2|+\de_1^{\al/2} \de_2^{-1+\al/2}] e^{(1+\al)\int_{T_0}^t Q_1 ds}.
\end{align*}
Using the second estimate from Lemma \ref{lemExponential} the claim follows for the case $T_1\leq t\leq \minT$. The calculation for $t\leq T_1$ is similar (and slightly simpler).

Next, using again Lemma \ref{lemFor_c_b} with $\ga_1=\ga_2= \al/2$ and Lemma \ref{lemExponential},
\begin{align*}
e^{-A} |c|&\leq e^{-A}X_2^{1-\al} [M + X_1^{\al/2} d(\bX)^{-1+\al/2}]\\
&\leq g\, e^{-\al \intto Q_2~ds} X_2(T_0)^{1-\al} \left[M + \de_1^{\al/2} d(\bX)^{-1+\al/2}\right].
\end{align*}
Hence
\begin{align*}
\intto e^{-A} |c|\intso e^{2 A} |b| &\ale g\, X_2(T_0)^{1-\al} X_1(T_0)^{1-\al}
\left[M|\log \de_2|+\de_1^{\al/2}{\de_2}^{-1+\al/2}\right]\\
&\qquad\times \intto e^{ \intso Q_2\,d\tau}[M + \de_1^{\al/2} d(\bX)^{-1+\al/2}]~ ds.
\end{align*}
We continue to estimate the last integral:
\begin{align*}
&\intto e^{ \intso Q_2\, d\tau} [M + \de_1^{\al/2} d(\bX)^{-1+\al/2}]~ ds\\
=&\,M\intto e^{ \intso Q_2\, d\tau} ds + \intto e^{ \intso Q_2\, d\tau} \de_1^{\al/2} d(\bX)^{-1+\al/2}~ ds\\
\le&\,M\intto\frac{Q_2}{Q_2} e^{ \intso Q_2\, d\tau} ds + e^{ \intto Q_2\, d\tau} \de_1^{\al/2} \intto  d(\bX)^{-1+\al/2}~ ds\\
\ale&\,e^{ \intto Q_2\, ds} \left[M+ \de_1^{\al/2}\de_2^{-1+\al/2}\right]
\end{align*}
where we used the familiar splitting at $T_1$. Thus, finally we get 
\begin{align*}
\intto e^{-A} |c|\intso e^{2 A} |b|&\ale g\, X_1(T_0)^{1-\al}
\left[M|\log \de_2|+\de_1^{\al/2}{\de_2}^{-1+\al/2}\right]\\
&\qquad\times [M + \de_1^{\al/2} \de_2^{-1+\al/2}] e^{ \intto Q_2\, ds} X_2(T_0)^{1-\al}.
\end{align*}
It remains to apply key Lemma \ref{keyLemma} to estimate the factor $e^{ \int_{T_0}^{t} Q_2 \,ds} X_2(T_0)^{1-\al}$,
which is less than
\begin{align*}
&\de_2^{1-\al} e^{ \int_{T_0}^t Q_2\, ds}
e^{-(1-\al) \int_{T_0}^{t} Q_2\, ds}
\le \de_2^{1-\al} e^{ \al \int_{T_0}^t Q_2\, ds}=\de_2^{1-2\al}\de_2^{\al} e^{ \al \int_{T_0}^t Q_2\, ds}.
\end{align*}
Now observe that the expression $\left[M|\log \de_2|+\de_1^{\al/2}{\de_2}^{-1+\al/2}\right] [M + \de_1^{\al/2} \de_2^{-1+\al/2}]\de_2^{1-2\al}$ is a harmless factor since $\al<1/4$.

\end{proof}
\begin{lem}\label{lem11} For sufficiently small $\de_2$ and $t\in[T_0,\minT]$ we have the following inequalities.
\begin{align}
\label{estv1}&v_1(t) \le g \left[Q_2 + M X_2^{1-\al}|\log d(\bX)| + \de_2^{1-\al} d(\bX)^{-1+\al}\right] e^{-\intto Q_2\, ds},\\
\label{estv2}&v_2(t) \leq g\, X_1(T_0)^{1-\al} \left[M|\log d(\bX)|+\de_1^{\al/2} d(\bX)^{-1+\al/2}\right]  e^{\al \intto Q_2\, ds},\\
\label{estintv1}&\intto v_1e^{\al\intso Q_2~d\tau}~ds \le g\\
\label{estintv2}&\intto v_2~ds \le g\,e^{\al \intto Q_2 \,ds}\\
\label{estintv22}&\intto v_2 e^{\al \intso Q_2 \,d\tau}\,ds \le g\, e^{2 \al \intto Q_2\, ds}, \\
\label{estintv1f1}&\intto v_1 f_1~ds \le g \,\de_1^{1-\al}\de_2^{\al},\\
\label{estintv2f1}&\intto v_2 f_1~ds \le g\, \de_1^{1-\al} \de_2^{\al/2}X_1(T_0)^{1-\al} e^{2\al\intto Q_2~d\tau},\\
\label{estintv1f2}&\intto v_1 f_2~ds \leq g X_2(T_0)^{1-\al},
\end{align}
where $g = g(\al, \beta, \bde, M)$ is a harmless factor.
\end{lem}
\begin{proof}
The estimates \eqref{estv1}-\eqref{estintv2} follow from Lemma \ref{lemFor_c_b}, Lemma \ref{lemExponential}, Proposition \ref{propSomeEst} and the usual splitting of the interval of integration into $[T_0,T_1]$ and $[T_1,\minT]$ . \eqref{estintv22} follows easily from \eqref{estintv2} and Lemma \ref{lemExponential}.

Using Lemma \ref{lemf1} and Lemma \ref{lemFor_c_b} we get
\begin{align*}
v_1 f_1 (s)& \leq  g\, \de_1^{1-\al}\de_2^{\al}\left[Q_2 + M X_2^{1-\al}|\log d(\bX)| + \de_2^{1-\al} d(\bX)^{-1+\al}\right] e^{(-1+\alpha) \intso Q_2\, d\tau}.
\end{align*} 
Note how the exponential growth of the factor $f_1$ was cancelled by the exponential factor in $v_1$.
By integrating, we get: 
\begin{align*}
\int_{T_0}^{t} v_1 f_1 ~ds&\le g\, \de_1^{1-\al}\de_2^{\al} \int\limits_{T_0}^{t} e^{(-1+\al) \intso Q_2 ~d\tau} \left[ Q_2 + M \de_2^{1-\al}|\log d(\bX)| + \de_2^{1-\al}d(\bX)^{-1+\al}\right]~ds \\
&\le g \,\de_1^{1-\al}\de_2^{\al}\left[M \de_2^{1-\al}|\log \de_2| +1 \right]
\end{align*}
giving \eqref{estintv1f1} since the factor in square brackets is harmless and can be absorbed into $g$. 

Proceeding analogously to prove \eqref{estintv2f1} we find
\begin{align*}
v_2f_1(s)&\le \,g \de_1^{1-\al} \de_2^{\al}X_1(T_0)^{1-\al} e^{2\al\intso Q_2~d\tau} \left[
M|\log d(\bX)| + \de_1^{\al/2} d(\bX)^{-1+\al/2}\right]
\end{align*}
which after integration from $T_0$ to $t$ can be estimated as follows:
\begin{align*}
\intto v_2f_1~ds &\le g\, \de_1^{1-\al} \de_2^{\al}X_1(T_0)^{1-\al} e^{2\al\intso Q_2~d\tau} \left[
M|\log \de_2| + \de_1^{\al/2} \de_2^{-1+\al/2}\right]\\
 &\le g\, \de_1^{1-\al} \de_2^{\al/2}X_1(T_0)^{1-\al} e^{2\al\intto Q_2~d\tau} \left[
M\de_2^{\al/2}|\log \de_2| + \de_1^{\al/2} \de_2\right].
\end{align*}
Again the factor in square brackets can be absorbed into $g$ giving \eqref{estintv2f1}.

\eqref{estintv1f2} is derived using the same techniques.

\end{proof}
\begin{lem}\label{lemEstH}
Along a particle trajectory, for $T_0\leq t \leq \minT$,
\begin{align}\label{lemEstH_eq3}
(H f_0)(t)&\leq g \|f_0\|_{\infty}(t)e^{\al \intto Q_2~ ds}
\end{align}
holds, where $\|f_0\|_{\infty}(t) =\sup\{|f_0(s)|\,:\,s\in [T_0, t]\}$.
\end{lem}
\begin{proof}
Using Lemmas \ref{lemf1}, \ref{lem11}, we get
\begin{align}
f_0+f_1 \expo{\intto v_1 f_1}\intto v_1 f_0& \le g\,\|f_0\|_{\infty}(t) e^{\al \intto Q_2\, ds} \label{lemEstH_eq1}.
\end{align}
Recall that products and exponentials of harmless factors are harmless, too.

Next, using Lemma \ref{lemf2f0}, \ref{lemf1} and \ref{lem11}
\begin{align*}
f_2+f_1 \expo{\intto v_1 f_1}\intto v_1 f_2& \le g\, X_2(T_0)^{1-\al}e^{\al \intto Q_2~ ds}\\
&= g\, X_2(T_0)^{1-2\al}X_2(T_0)^{\al}e^{\al \intto Q_2~ ds}\\
&\le g\, X_2(T_0)^{1-2\al}\de_2^{\al}
\end{align*}
with the key Lemma \ref{keyLemma} in the last step to cancel of $e^{\al \intto Q_2 ~ds}$ using the factor $X_2(T_0)^{\al}$. 

So for $v_2 \left[f_2+f_1 \expo{\intto v_1 f_1}\intto v_1 f_2\right]$ we obtain the upper bound
\begin{align*}
v_2 g  X_2(T_0)^{1-2 \al}\de_2^{\al} &\le g\, e^{\al \intto Q_2} X_2(T_0)^{1-2 \al}\de_2^{\al} X_1(T_0)^{1-\al}\left[M|\log d(\bX)|+\de_1^{\al/2} d(\bX)^{-1+\al/2}\right] \\
&\leq g\, \de_2^{2\al} X_2(T_0)^{1-3 \al}\de_1^{1-\al} \left[M|\log d(\bX)|+\de_1^{\al/2} d(\bX)^{-1+\al/2}\right] 
\end{align*}
using the key Lemma \ref{keyLemma} again to cancel $e^{\al \intto Q_2}$. Thus we see that 
\begin{align*}
\expo{\intto v_2 \left[f_2+f_1 \expo{\intso v_1 f_1} \intso v_1 f_2\right]} \le g
\end{align*}
Finally, by \eqref{lemEstH_eq1} and Lemma \ref{lem11}
\begin{align*}
&\intto v_2 \left[f_0+f_1 \expo{\intso v_1 f_1}\intso v_1 f_0\right] \le\\
&\,\,g \|f_0\|_{\infty}(t) \intto v_2 e^{\al \intso Q_2 \,d\tau} ds \le g \|f_0\|_{\infty}(t) e^{2\al \intto Q_2\, ds}. 
\end{align*} 
Thus, in total we get
\begin{align}
\nn f_0+f_1 \expo{\intto v_1 f_1}\intto v_1 f_0 &\le g \|f_0\|_{\infty}(t) e^{\al \intto Q_2\, ds}\\
\label{H0part2}f_2+f_1 \expo{\intto v_1 f_1}\intto v_1 f_2 &\le g X_2(T_0)^{1-2\al}\de_2^\al\\
\label{H0part3}\expo{\intto v_2 \left[f_2+f_1 \expo{\intso v_1 f_1} \intso v_1 f_2\right]} &\le g\\
\label{H0part4}\intto v_2 \left[f_0+f_1 \expo{\intso v_1 f_1}\intso v_1 f_0\right] &\le g \|f_0\|_{\infty}(t) e^{2\al \intto Q_2~ ds}.
\end{align}
Note that 
\begin{align*}
\eqref{H0part2}\times\eqref{H0part4}\le g \, X_2(T_0)^{1-4\al}\de_2^{4\al}\|f_0\|_{\infty}(t)\le g
\end{align*}
using again the key Lemma to get rid of the factor $e^{2\al \intto Q_2}$, and in the
very last step we used $\al\in (0, 1/4)$ and Lemma \ref{lemf2f0}.
Combining the inequalities \eqref{H0part2}-\eqref{H0part4} gives
\begin{align*}
\left(f_2+f_1 \expo{\intto v_1 f_1}\intto v_1 f_2\right) \\
\times \expo{\intto v_2 \left[f_2+f_1 \expo{\intso v_1 f_1}\intto v_1 f_2\right]} \\
 \times \intto v_2 \left[f_0+f_1 \expo{\intso v_1 f_1}\intto v_1 f_0\right]\\
\le g X_2(T_0)^{1-4 \al}\de_2^{4\al} \|f_0\|_{\infty}(t)  \le g \|f_0\|_\infty(t)
\end{align*}

In view of \eqref{eq_H}, \eqref{lemEstH_eq3} now follows.
\end{proof}
\subsection{Proof of the main technical theorem}
\begin{proof}[Proof of theorem \ref{mainTechnicalThm}]
At time $t=T$, any
$x\in D$ is occupied by a particle, i.e. $x = \bX(T)$ for some particle trajectory. Let us write 
\begin{align*}
\d_{x_j} \om(\bX(t), t) = \xi_j(t) 
\end{align*}
along that particle trajectory, and so by \eqref{ode_ineq5},
\begin{align*}
|\xi_1(t)|&\leq (H f_0)(t) + e^{A} \intto |a| e^{-A} (H f_0)(s)~ ds= (H f_0)(t) + e^{A} \intto v_1 (H f_0)(s)~ds.
\end{align*}
First note that by Lemmas \ref{lemEstH}, \ref{lemf2f0}, \ref{lem11} and \ref{lemExponential}
\begin{align*}
e^{A} \intto v_1(s) (H f_0)(s)~ ds &\le e^{A} g\, \|f_0\|_{\infty}(t) \intto v_1e^{\al\intso Q_2~d\tau}~ds \le  g\, R X_2(T_0)^{1-\al} e^{A}\\
&\le  g\, R e^{-(1-\al)\intto Q_2~ds} e^{\intto Q_2~ds}\le g\, Re^{\al\intto Q_2~ds}.
\end{align*} 

Moreover again by Lemmas \ref{lemEstH}, \ref{lemf2f0}
\begin{align*}
(H f_0)(t) &\leq  g\, Re^{\al \intto Q_2~ds}.
\end{align*}
This results in
\begin{align}\label{eq_720}
|\xi_1(t)| &\leq g\, Re^{\al \intto Q_2~ds}.
\end{align}
Next we estimate $|\xi_1(t)|X_1(t)^\al$.
First we use \eqref{eq_720} and 
insert $X_1(t)$ from \eqref{X2}. Then 
Lemma \ref{lemExponential} allows us to replace $Q_1$ by $Q_2$ in 
the one of the arguments of the exponential function, so we get the estimate
\begin{align}
\begin{split}
|\xi_1(t)|X_1(t)^\al &\le g\, R X_1(t)^\al e^{\al \intto Q_1~ ds}
\le g\,R\de_1^\al e^{-\al \intto Q_1~ ds}e^{\al \intto Q_2~ ds}\\
&\le g\,R\de_1^\al e^{-\al \intto Q_2~ ds}e^{\al \intto Q_2~ ds}
\le g\,R\de_2^\al
\end{split} 
\label{proofMainTechnicalThmEq5} 
\end{align}
since $\de_1<\de_2$. In fact, this was the most critical estimate in the whole proof of the main result, since the dangerous factor $e^{A}$ was barely cancelled. 

We now derive a similar estimate for $|\xi_2(t)|X_2(t)^\al$. From the second line of \eqref{ode_ineq5} and the assumptions on initial conditions, 
\begin{align}\label{proofMainTechnicalThmEq3}
|\xi_2(t)| &\leq R e^{-A} + e^{-A} \intto v_2 \,H f_0~ ds+ e^{-A} \intto e^{2 A} |b| \intso v_1\, H f_0~ds. 
\end{align}
By Lemma  \ref{lemEstH} and \ref{lemf2f0} $Hf_0(s)$ has the upper bound \begin{align*}
&R\,g\,X_2(T_0)^{1-\al}e^{\al\intso Q_2~d\tau}.
\end{align*} 
Therefore the second summand can be estimated as follows
\begin{align*}
e^{-A} \intto v_2 \,H f_0~ ds&\le e^{-A}R\,g\,\intto v_2 e^{\al\intso Q_2~d\tau}~ds
\le R\,g\,e^{-A}e^{2\al\intto Q_2~ds}\\
&\le R\,g\,e^{(-1+2\al)\intto Q_2~ds}\le g\,R,
\end{align*}
where we also used $X_2(T_0)\le\de_2$, \eqref{estintv22}, Lemma \ref{lemExponential} and $\al<1/4$.

\noindent For the third summand of \eqref{proofMainTechnicalThmEq3} we use the upper bound for $Hf_0$ again:
\begin{align*}
 e^{-A} \intto e^{2 A} |b| \intso v_1\, H f_0~ds&\le g\,R\,X_2(T_0)^{1-\al} e^{-A} \intto e^{2 A} |b| \intso v_1\,e^{\al\int_{T_0}^\tau Q_2~d\zeta}~d\tau ~ds\\
&\le g\,R\,\de_2^{1-\al} e^{-A} \intto e^{2 A} |b| ~ds.
\end{align*}
In the last estimate \eqref{estintv1} was used. Now note that by \eqref{lemf1_claim} 
\begin{align*}
 \intto e^{2 A} |b| ~ds&\le g\, e^{(1+\al)\intto Q_2~ds}.
\end{align*}
After combining this with the $e^{-A}$ factor we see that we can bound the third summand by $g\,R$, i.e.~$|\xi_2(t)|\leq g\,R$. This means that
\begin{align}\label{proofMainTechnicalThmEq4}
|\xi_2(t)|X_2^{\al} \leq g R \de_2^{\al}.
\end{align} 
The inequalities \eqref{proofMainTechnicalThmEq4} and \eqref{proofMainTechnicalThmEq5} imply
\begin{align*}
M_{D}(T) \leq g(\al, \beta, \bde, M) R \de_2^{\al} + \|\om\|_\infty =: \cN(R, \al, \beta, \bde, M).
\end{align*}
It remains to show that $\cN$ is a harmless nonlinearity. Therefore, let $\al, \beta, R$ be given.
Recall that $g$ has the property that $g(\al, \beta, \de_2^p, \de_2, R)$ is bounded as $\de_2 \to 0$ for all $p>p_0$
with some $p_0 > 0$. Hence
\begin{align*}
g(\al, \beta, \de_2^p, \de_2, R) R \de_2^{\al} +\|\om\|_\infty< R
\end{align*}
 for sufficiently small $\de_2 > 0$ and $R > \|\om\|_\infty$.
\end{proof}

\section{Proof of the main result}

\begin{proof}[Proof of Theorem \ref{main}]
Let 
$\al\in (0, \frac{1}{4})$, $0<\de_3<1/2$ and $R > \|\om\|_\infty$ be a given nonnegative number.
Fix small positive $\de_1, \de_2$ such that the following set of inequalities hold true:
\begin{align}
\de_1, \de_2 \leq \rho, ~\beta_0 - A |\bde|^{1-\al} R \ge \half \beta_0. \label{mainThm_ineq3}\\
M_{D}(0) < R, ~\left|\frac{\d \om_0(x)}{\d x_1}\right|\le R x_2^{1-\al}, ~\left|\frac{\d \om_0(x)}{\d x_2}\right|\le R \quad (x\in D)\label{mainThm_ineq4}
\end{align}
where $A, \beta_0, \rho$ are the numbers from the Definition \ref{defhyperbolicflow} of the hyperbolicity of the flow.
Note that the box can be chosen so small that that \eqref{mainThm_ineq4} holds. This a 
consequence of $\frac{\d \om_0}{\d x_2}(0, x_2)=\frac{\d \om_0}{\d x_1}(x_1, 0)=0$ and the
$C^2$-smoothness of $\om_0$.

\emph{Claim: } If the box $\dD$ satisfies the controlled feeding conditions with parameter $R$,
then we have the bound
\begin{align}\label{mainThm_ineq1}
M_{D}(t) \leq R\quad (t\in [0, \infty)).
\end{align} 

Assume \eqref{mainThm_ineq1} is not true for all times, i.e.~there is a time $\til T$ such that $M_{D}(\til T) > R$.  Since the solution $\om$ is sufficiently smooth in time by assumption, $M_D(t)$ is a continuous function of $t$. Because
$M_{D}(0) < R$, by the intermediate value theorem, 
there exists a time $T\in(0,\til T)$ such that $M_{D}(t) < R$ holds
on $[0, T)$ and $M_{D}(T)=R$. Observe also that automatically $M_{\dD}(t) \leq R$ for $t\leq T$.

Now note that by \eqref{mainThm_ineq3},
the flow is $\half \beta_0$-hyperbolic in the box $D$ on the time interval $[0, T]$.
This can be seen as follows. Because of \eqref{condDelta} and the feeding conditions,
 $M(x,t)\le M_{\dD}(t)\le R$ for all $x\in D$ and $t\in[0,T]$. Thus by \eqref{mainThm_ineq3}
\begin{align*}
Q_i(x,t)&= Q_i(x,t)+A|x|^{1-\al}M(x,t)-A|x|^{1-\al}M(x,t)\ge \beta_0-\half \beta_0.
\end{align*}

Upon shrinking $\de_1$ and $\de_2$ further (if necessary) and using \eqref{mainThm_ineq4} we can achieve that the assumptions of Theorem \ref{mainTechnicalThm} are satisfied, the arguments in Section \ref{sec:main} hold and for the harmless nonlinear function $\cN$ from theorem \ref{mainTechnicalThm} the following inequality is true:
\begin{align}
\cN(R, \al, \half\beta_0, \de_1, \de_2, R) &< R. \label{mainThm_ineq2}
\end{align}
From now on $\bde$ is fixed. 

On the one hand, on $[0, T]$, we have $M_{D}(t) \le R$.
But applying Theorem \ref{mainTechnicalThm} with $K=R$ and \eqref{mainThm_ineq2}, we get
\begin{align*}
M_{D}(T) \leq \cN(R, \al, \half\beta_0, \de_1, \de_2, R) < R,
\end{align*}
a contradiction. This proves our claim \eqref{mainThm_ineq1}.

Now we prove the exponential bound on the gradient growth.
At an arbitrary time $t \geq 0$, each $x \in D$ is occupied by a particle $\bX(t)$ that has entered the box at some earlier
time $T_0$, or $T_0=0$ if the particle started in $D$ at $t=0$. The same calculation leading to \eqref{proofMainTechnicalThmEq5} gives
\begin{align*}
\left|\frac{\d \om}{\d x_1}(\bX(t), t)\right|&\leq g\, R e^{\al \intto Q_2~ ds}.
\end{align*} 
for all $t\in [T_0, T_e]$.
We apply now Lemma \ref{lemUpperboundQ}:
\begin{align*}
\intto Q_2~ ds &\ale  (\|\om\|_{\infty}+ R(|\bde|+\de_3)^{1-\al}) (t-T_0) + \|\om\|_{\infty} \intto |\log d(\bX)| ~ds . 
\end{align*}
The integral containing the logarithmic term can be estimated using the familiar splitting at $T_1$ and gives
\begin{align*}
\intto |\log d(\bX)|~ds &\ale |\log \de_2|  (t - T_0)+|\log \de_2|. 
\end{align*}
Thus, finally, 
\begin{align*}
\left|\frac{\d \om}{\d x_1}(\bX(t), t)\right| \leq C(\al, \beta, \bde, \de_3, R,\|\om\|_\infty) e^{\al(C\|\om\|_{\infty}+ R(|\bde|+\de_3)^{1-\al}+|\log \de_2|) t}.
\end{align*}
The derivative in $x_2$-direction is bounded by $g\,R$ as seen in the proof of Theorem \ref{mainTechnicalThm}.
This concludes the proof of Theorem \ref{main}.
\end{proof}

\begin{remark}
Our main result remains valid if instead of $\om_0 \in C^2$ we only 
assume $\om_0 \in C^{1, \ga}$ with $\ga \in (3/4, 1)$ (note that
\eqref{mainThm_ineq4} is still true in this case). Recall that in  
\cite{Zlatos} a solution in $C^{1, \ga}$ was constructed such that
the gradient growth close to the origin is at least exponential.
This allows us to state the following interesting \emph{conditional result}:

\begin{cor}
Suppose a solution $\om$ in $C^{1, \ga}$ as in
 \cite{Zlatos} exists that satisfies the feeding conditions
in some box. Then exponential gradient growth near the origin
is optimal in the class of $C^{1, \ga}$-solutions.
\end{cor}

\end{remark}

\begin{remark}\label{aboutTimeDepFeeding}
At this point, we address the difficulties in applying the techniques developed here to the case of 
time-dependent feeding. Assume for the sake of the discussion that in \eqref{eqfeeding1} we replace
the constant $R$ by
\begin{align*}
R(t) = R_1 t + R_2
\end{align*} 
with positive constants $R_1, R_2$. As a direct consequence, $f_0$ grows in time. The corresponding 
inequality in Lemma \ref{lemf2f0} will roughly read
\begin{align*}
f_0(t) \ale g R(T_0) X_2(T_0)^{1-\al}.
\end{align*}
This produces, for example, via \eqref{lemEstH_eq3} a growth in our bounds for $M_{\dD}(t)$, so that 
$M_{\dD}(t)$ cannot be bounded by a time-independent constant anymore. The consequences are as follows:
\begin{itemize}
\item The hyperbolicity condition \eqref{Hyperbolicity} is no longer sufficient to stabilize the flow. It has to be 
considerably strengthened, for example by requiring the flow to be at least $\beta$-hyperbolic from the outset. Lemma
\ref{lemExponential} does not seem to go through, so one may have to strengthen the condition even further, e.g. by
requiring
\begin{align}\label{strongHyp}
a(x, t) \geq \beta > 0.
\end{align}
This, however, has the disadvantage that we have no sufficient condition on the initial data to
justify the validity of \eqref{strongHyp}.
\item A destabilizing effect is also felt in all estimates of Lemma \ref{lem11}, especially when 
the now time-dependent bound for $M_{\dD}(t)$ is integrated along a particle trajectory with
a growing factor (e.g. $\exp(\al \int_{T_0}^t Q_s d\tau)$). This leads to worse estimates
for the quantities $f_1, f_2$ and hence for $H f_0$, which are a vital part of the main argument.
\end{itemize}
In this paper, we leave the problem of finding a suitable treatment for time-dependent feeding open.
\end{remark}

\section{Acknowlegdements}
The authors cordially thank A. Kiselev for suggesting the problem and a great number of helpful discussions,
as well the anonymous reviewers for their careful reading of the manuscript and helpful comments.
VH would like to express his gratitude to the Deutsche Forschungsgemeinschaft (German Research Foundation),
without whose financial support (FOR HO 5156/1-1 and FOR HO 5156/1-2) the present research could not have been undertaken.
VH also acknowledges partial support by NSF grant NSF-DMS 1412023.

\section{Appendix}

\subsection{Appendix A}

We use the Kronecker symbol
\begin{align*}
\de_{ij}&=\left\{
\begin{array}{cc}
1& \text{ if }i=j\\
0& \text{ if }i\neq j
\end{array}
\right. .
\end{align*}

\begin{prop}\label{prop:A1}
For all $x,y\in\ov,~x\neq y$ the following estimates hold.
\begin{align*}
\left|G_i^k(x,y)\right|&\ale|y-x|^{-1}x_i^{-1} \\
\left|\frac{\d G_i^k}{\d x_j}(x,y)\right|&\ale|y-x|^{-3}
\end{align*}
$(i, k = 1, 2)$.
\end{prop}
The proofs are straightforward calculations based on the identities in Appendix B, and the reflection inequalities:
\begin{align}\label{reflectionIneq}
|y-\til x| \geq |y-x|,\quad|y-\bar x| \geq |y-x|,\quad |y+x| \geq |y-\bar x|,\quad |y+x| \geq |y-\til x|
\end{align}
holding for $x, y\in \ov$.
Also, use the obvious inequalities
\begin{align*}
x_2 \leq |y-\bar x|, x_1 \le |y-\til x|.
\end{align*}

We observe some useful relations for the kernels $G_i^k$ and their derivatives. Let $G$ stand for any $G_i^k$ and let
\begin{align*}
\Om_x = (-x_1, 1-x_1)\times(-x_2, 1-x_2).
\end{align*}
$G$ has the form $G(x, y) = \tG(y-x, x, y)$, where $\tG(z, \eta, \mu)$ is 
smooth provided $\eta, \mu\in (0,1)^2, z\in \Om_x \setminus \{0\}$. For example, if $G=G_1^1$ then
\begin{align*}
\tG(z, \eta, \mu)=\frac{\mu_1 z_1}{|z|^2 |\mu-\til \eta|^2}.
\end{align*}
Note that for $x\neq y, x, y\in (0, 1)^2$,
\begin{align}\label{dev_eq4}
\begin{split}
(\d_{x_j} G)(x, y) = (\d_{\eta_j}\tG)(y-x, x, y) - (\d_{z_j}\tG)(y-x, x, y)\\
(\d_{y_j} G)(x, y) = (\d_{\mu_j}\tG)(y-x, x, y) +(\d_{z_j}\tG)(y-x, x, y)
\end{split}
\end{align}
so that
\begin{align*}
(\d_{x_j} G)(x, y) = -(\d_{y_j} G)(x, y) +(\d_{\eta_j}\tG)(y-x, x, y) + (\d_{\mu_j}\tG)(y-x, x, y).
\end{align*}
Moreover, we always have
\begin{align}\label{dev_ineq1}
|\tG(z, x, y)|, \left|\frac{\d \tG}{\d \eta_j}(z, x, x+z)\right|, \left|\frac{\d \tG}{\d \mu_j}(z, x, x+z)\right|\leq C(\eta)|z|^{-1}.
\end{align}
where $C(\eta)$ is uniformly bounded if $\eta$ varies in a compact subset of $(0, 1)^2$.

\begin{prop}\label{prop:A7}
\begin{align*}
\frac{\d G_i^k}{\d x_j}=-\frac{\d G_i^k}{\d y_j}+x_i^{-2}\de_{ij}\cO(|y-x|^{-1})
\end{align*}
\end{prop}

\begin{proof}
This is a tedious, but straighforward estimation using the identities of Appendix B
and the reflection inequalities. 
\end{proof}


\begin{prop}[Derivatives of $Q_i$]\label{prop:A2}
\begin{align*}
\frac{\d Q_i}{\d x_j}&=
c_0\pvint_{[0,1]^2}\left[\frac{\d G_i^1}{\d x_j}+\frac{\d G_i^2}{\d x_j}\right] \om(y)~dy\\
& - \om(x)\lim_{\de \to 0^+}\int_{\d B(x, \de)} G_i^i\cdot\nu_j~d\si+\frac{\d Q_i^r}{\d x_j}
\end{align*}
\end{prop}

\begin{proof}
Write  $(G_i^1+G_i^2)(x, y):=G(x, y)$. $G$ has again the form $G(x, y) = \tG(y-x, x, y)$, where $\tG(z, \eta, \mu)$ is 
smooth provided $\eta, \mu\in (0,1)^2, z\in \Om_x \setminus \{0\}$. Also \eqref{dev_eq4}, \eqref{dev_ineq1} hold for $\tG$.
Now 
\begin{align}
 \label{dev_eq3}\frac{\d}{\d x_j} \int_{\Om_x} \tG(z, x, x+z) \om(x+z)~dz= \int_{\Om_x} \tG(z, x, x+z) \frac{\d \om}{\d z_j}(x+z)~dz\\
+\int_{\Om_x} \d_{x_j}(\tG(z, x, x+z)) \om(x+z)~dz\nn \\
- \int_{\d \Om_x} \tG(z, x, x+z) \om(x+z)\nu_j~d\si\nn
\end{align}
where $\nu_j$ denotes the $j$-th component of the unit outer normal. This is a standard differentiation result (note the
bounds \eqref{dev_ineq1}).

Now consider the integral in the line \eqref{dev_eq3}, exclude the singularity and integrate by parts:
\begin{align}\label{dev_eq5}
\int_{\Om_x} \tG(z, x, x+z) \frac{\d \om}{\d z_j}(x+z)~dz =- \int_{\Om_x\setminus B(0, \de)} \d_{z_j}(\tG(z, x, x+z)) \om(x+z)~dz\\
+ \int_{\d \Om_x} \tG(z, x, x+z) \om(x+z)\nu_j~d\si \nn\\
- \int_{\d B(0, \de)} \tG(z, x, x+z) \om(x+z)\nu_j~d\si\nn
\end{align}
Observe that by \eqref{dev_eq4},
\begin{align*}
-\d_{z_j}(\tG(z, x, x+z))+\d_{x_j}(G(z, x, x+z)) = (\d_{x_j}G)(z, x, x+z).
\end{align*}
So combining \eqref{dev_eq3} and \eqref{dev_eq5}, we finally get
\begin{align*}
\frac{\d}{\d x_j} \int_{\Om_x} \tG(z, x, x+z) \om(x+z)~dz =  -\int_{\d B(0, \de)} \tG(z, x, x+z) \om(x+z)\nu_j~d\si\\
+ \int_{\Om_x}  (\d_{x_j}G)(z, x, x+z) \om(x+z)~dz + \int_{B(0, \de)} \d_{x_j}(\tG(z, x, x+z)) \om(x+z)~dz.
\end{align*}
Replacing $x+z$ by $y$ and sending $\de\to 0$ yields the statement.
\end{proof}

We define
\begin{align*}
d_1(x) := \min\{x_1, x_2\}
\end{align*}
which is the distance of the point $x$ to the coordinate axes.
Observe also that
\begin{align}\label{relXY}
\half x_r \le y_r \le \frac{3}{2}x_r 
\end{align}
for $y\in B(x, \half d_1(x)), r=1, 2$. For the entire appendix, we shall write that
$M = M_{\dD}$, i.e.
\begin{align*}
\left|\frac{\d\om}{\d x_j}(x)\right| \le M x_j^{-\al}\quad (x \in \dD, j=1, 2)
\end{align*}
holds, implying also the inequalities
\begin{align}\label{estOm}
|\om(x)| \ale M x_j^{1-\al} \quad (x\in \dD, j=1, 2)
\end{align}
(by the fact that $\om$ vanishes identically on the coordinate axes).

\begin{prop}\label{prop:A6}
For $i\neq j$, we have
\begin{align*}
\left|\frac{\d (G_{i}^1+G_{i}^2)}{\d x_j}\right|  \ale  x_i^{-\ga_1-\ga_2} x_j^{\ga_2} |y-x|^{-(3 - \ga_1)}.
\end{align*}  
where $\ga_1, \ga_2\in \R, 0\le \ga_1+\ga_2\leq 1$.
\end{prop}
\begin{proof}
We do the proof in the case $i=2, j=1$, the other case being analogous.
The proof of the proposition is based on a cancellation property of the kernels $G_2^1$ and $G_2^2$ and 
requires quite tedious computations. First calculate the sum of $\d_{x_1} G_2^2$ and $\d_{x_1} G_2^1$ and see that
it can be grouped into three the expressions
\begin{align*}
\frac{2 y_2(y_1-x_1)^2}{|y-x|^4|y-\bar x|^2}-\frac{2 y_2(y_1+x_1)^2}{|y+x|^2|y-\til x|^4} &= (A)\\
\frac{2 y_2(y_1-x_1)^2}{|y-x|^2|y-\bar x|^4}-\frac{2 y_2(y_1+x_1)^2}{|y-\til x|^2|y + x|^4} &= (B)\\
\frac{y_2}{|y - \til x|^2|y+x|^2}-\frac{y_2}{|y-x|^2|y-\bar x|^2} &= (C)
\end{align*}
These can be further written as
\begin{align*}
\frac{2 y_2(y_1-x_1)^2}{|y-\bar x|^2}[|y-x|^{-4}-|y-\til x|^{-4}]+\frac{2 y_2}{|y-\til x|^4}\left(
\frac{(y_1-x_1)^2}{|y-\bar x|^2} - \frac{(y_1+x_1)^2}{|y+x|^2}\right) \\= (1)+(2)\\
\frac{2 y_2}{|y-\bar x|^4}\left[\frac{(y_1-x_1)^2}{|y-x|^2}-\frac{(y_1+x_1)^2}{|y-\til x|^2} \right]
+\frac{2 y_2}{|y-\til x|^2}\left[\frac{(y_1+x_1)^2}{|y-\bar x|^4}-\frac{(y_1+x_1)^2}{|y+x|^4}\right] \\= (3)+(4)\\
\frac{y_2}{|y+x|^2}\left[|y-\til x|^{-2}-|y - x|^{-2}\right]+\frac{y_2}{|y-x|^2}[|y+ x|^{-2}-|y - \bar x|^{-2}]
\\= (5)+(6)
\end{align*}
Let us estimate expression (1). Using $|y-\til x|^2 - |y-x|^2 = 4 x_1 y_1$ and the relations
$y_2 \leq |y-\bar x|, (y_1-x_1)^2\leq |y-x|^2, y_1\leq (y_1+x_1), |y-x|\le|y-\til x|$, we arrive at
\begin{align*}
|(1)| \ale \frac{x_1}{|y-\bar x||y-\til x||y-x|^{2}}.
\end{align*}
Write $\ga = \ga_1+\ga_2$ and noting that $|y-\bar x|\geq x_2^{\ga}|y-\bar x|^{1-\ga}, |y-\til x|\geq x_1^{1-\ga_2}|y-\til x|^{\ga_2}$ 
and the reflection relations $|y-\til x|, |y - \bar x|\geq |y-x|$ for $y\in \ov$, we arrive at
\begin{align*}
|(1)| \ale x_2^{-\ga} x_1^{\ga_2} |y-x|^{-(3 - \ga + \ga_2)}.
\end{align*}

To estimate (2), we use the relation
\begin{align*}
|y+x|^2 (y_1 - x_1)^2 - |y-\bar x|^2 (y_1+x_1)^2 = - 4 x_1 y_1 (y_2+x_2)^2
\end{align*}
and similar estimations as above to arrive at
\begin{align*}
|(2)| \ale \frac{y_2 y_1 x_1 (y_2+x_2)^2}{|y-\til x|^4 |y-\bar x|^2 |y+x|^2} \ale \frac{x_1}{|y-\bar x||y-\til x|^3}\\
\ale x_2^{-\ga} x_1^{\ga_2} |y-x|^{-3 - \ga_2 + \ga}.
\end{align*}
(3)-(6) is mutatis mutandis the same.
\end{proof}

Figure \ref{fig2} illustrates the domains we need in the proof of the following propositions. 
\begin{figure}[htbp]
\begin{center}
\includegraphics[scale=1]{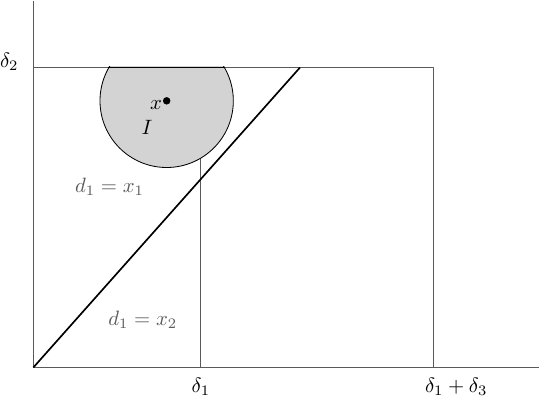}%
\end{center}
\caption{Domains of integration in Proposition \ref{prop:A3}.} \label{fig2}
\end{figure}

\begin{prop}\label{prop:A3} Let $I=B(x, \half d_1(x))\cap \dD$. Then
\begin{align*}
\left|\pvint_{I}\frac{\d(G_i^1+G_i^2)}{\d x_j}\om(y)~dy\right| \ale M x_i^{-\al}(1+\de_{j2}|\log d(x)|)
\end{align*}
\end{prop}

\begin{proof}
Let $0<\de<\half d_1(x)$ so small such that $B(x, \de)\subset \dD$. 
By Proposition \ref{prop:A7},
\begin{align}\label{proofOfPropA3eq1}
\begin{split}
&\left|\int_{I\setminus B(x, \de)}\sum_{k=1,2} \frac{\d G_i^k}{\d x_j}\om(y)~dy\right|\ale\left|\int_{I\setminus B(x, \de)}\sum_{k=1,2} \frac{\d G_i^k}{\d y_j}\om(y)~dy\right|+
\left| x_i^{-2} \int_{I\setminus B(x, \de)} \cO(|y-x|^{-1}) \right|
\end{split}
\end{align}
We distinguish the cases $i=j$ and $i\neq j$. First let $i=j$. Integration by parts gives
\begin{align*}
\begin{split}
\text{\eqref{proofOfPropA3eq1}}&\ale\left|\int_{I\setminus B(x, \de)}\sum_{k=1,2} G_i^k\frac{\d \om}{\d y_i}(y)~dy\right| + \left| x_i^{-2} \int_{I\setminus B(x, \de)} \cO(|y-x|^{-1}) \right|\\
&+\sum_{k=1,2} \left(\left|\int_{\d I} G_i^k  \om(y)\nu_i~d\si \right| + \left|\int_{\d B(x, \de)}G_i^k \om(y)\nu_i~d\si\right|\right).
\end{split}
\end{align*}
where $\nu=(\nu_1,\nu_2)$ is the unit outward pointing normal on $\d I$. We first take care of the integral over $\d I$.
Observe that for $x\in D$, $\d I$ is either a full circle is the union of a part of a circle and a flat part $\Sigma$. Hence
\begin{align*}
\int_{\d I} |G_i^k \nu_j|~d\si &\leq \int_{\Sigma} |G_i^k| \de_{2i} ~d\si +\int_{\d B(0, \half d_1(x))} |G_i^k|~d\si
\end{align*}
For all sufficiently small $\vareps>0$,
\begin{align*}
\int_{\Sigma} |G_i^k|~d\si &\le\int_{\Sigma\cap \{|y_1-x_1|\le \eps\}} |G_i^k|~dy_1+\int_{\Sigma\cap \{|y_1-x_1|\ge \eps\}} |G_i^k|~dy_1\\
&\ale\int_{\Sigma\cap \{|y_1-x_1|\le \eps\}} \frac{x_i^{-1}}{|y-x|}~dy_1
+ \int_{\Sigma\cap \{|y_1-x_1|\ge \eps\}} \frac{x_i^{-1}}{|y-x|}~dy_1\\
&\ale x_i^{-1}\int_{x_1-\vareps}^{x_1+\vareps} \frac{1}{|x_2-\de_2|}~dy_1
+ x_i^{-1}\int_{1>|x_1-y_1|>\vareps} \frac{1}{|y_1-x_1|}~dy_1\\
&\ale \frac{x_i^{-1}\vareps}{|x_2-\de_2|}+ x_i^{-1}\int_{\vareps}^1 \frac{1}{y_1}~dy_1\\
&\ale \frac{x_i^{-1}\vareps}{|x_2-\de_2|}+ x_i^{-1}|\log\vareps|.
\end{align*}
Here we used proposition \ref{prop:A1} again.
Choosing $\vareps=|x_2-\de_2|=d(x)$ we get
\begin{align*}
\int_{\Sigma} |G_i^k|~d\si
&\ale x_i^{-1}(1+|\log d(x)|).
\end{align*}
The other part is estimated by (using proposition \ref{prop:A1} again)
\begin{align*}
\int_{\d B(0, \half d_1(x))} |G_i^k|~d\si
&\ale
x_i^{-1}\int_{0}^{2\pi} |y-x|^{-1} d_1(x) ~d\varphi
\ale x_i^{-1}.
\end{align*}
Therefore we get for the integral over $\d I$, using \eqref{estOm} and \eqref{relXY}:
\begin{align*}
\left|\int_{\d I} G_i^k \om(y)\nu_i~d\si \right|&\ale \int_{\d I}  \left| G_i^k\nu_i \right|  \left|\om(y) \right|~d\si
\ale M \int_{\d I}  \left|G_i^k\nu_i \right|  y_i^{1-\al}~d\si\\
&\ale M x_i^{1-\al}\int_{\d I}  \left|G_i^k\nu_i \right| ~d\si
\ale  M x_i^{1-\al}x_i^{-1}(1+\de_{i2}|\log d(x)|).
\end{align*}
Similar estimates yield that the contribution from the integral over 
$\d B(x, \de)$ is $\ale M x_i^{-\al}$, with universal constants independent of $\de$.
For
\begin{align*}
x_i^{-2} \int_{I\setminus B(x, \de)} \cO(|y-x|^{-1}) |\om(y)| dy
\end{align*}
we obtain the upper bound $\ale M x_i^{-\al}$ by the same methods.

For the remaining integral we use \eqref{relXY}:
\begin{align*}
\left|\int_{I\setminus B(x, \de)}G_i^k\frac{\d \om}{\d y_i}(y)~dy\right|&\ale M x_i^{-\al}\int_{I\setminus B(x, \de)}\left|G_i^k\right|~dy
\ale M x_i^{-\al} \int_{I\setminus B(x, \de)}\left|y-x\right|^{-1}x_i^{-1}~dy\\
&\ale M x_i^{-\al} x_i^{-1}\int_{\de}^{d_1(x)}\frac{1}{\rho}\rho ~d\rho \ale M x_i^{-\al}x_i^{-1}\int_{\de}^{d_1(x)}\frac{1}{\rho}\rho ~d\rho\\
&\ale M x_i^{-\al} x_i^{-1}d_1(x).
\end{align*}
Since $d_1(x)\le x_i$ we get:
\begin{align*}
\left|\int_{I\setminus B(x, \de)}G_i^k\frac{\d \om}{\d y_j}(y)~dy\right|\ale M x_i^{-\al}.
\end{align*}
This concludes the case $i= j$.
 
For the case $i\neq j$, we have to use the cancellation provided by Proposition \ref{prop:A6} with $\ga_1=2, \ga_2=-1$.
\begin{align*}
\left|\int_{I\setminus B(x, \de)}\frac{\d(G_i^1+G_i^2)}{\d y_j}\om(y)~ \right|~dy&\ale M \int_{I\setminus B(x, \de)}\left|\frac{\d(G_i^1+G_i^2)}{\d y_j}\right|y_i^{1-\al}~ dy\\
&\ale  M x_i^{1-\al} \int_{I\setminus B(x, \de)} y_i^{-1}y_j^{-1}|y-x|^{-1}  ~dy\\
&\ale  M x_i^{1-\al}  x_i^{-1}x_j^{-1}\int_{\de}^{d_1(x)}\frac{1}{\rho}\rho  ~d\rho\\
&\ale  M x_i^{1-\al}  x_i^{-1}x_j^{-1}d_1(x) \ale  M x_i^{-\al},
\end{align*}
since $d_1(x)\ale x_j$.
\end{proof}
\begin{lem}\label{lem:A2} Let $\ga\in (0,1)$ and $x_1\geq 0$. Then
\begin{align*}
y_1^\ga\le|y_1-x_1|^\ga+x_1^\ga\quad (y_1\geq 0).\nn
\end{align*}
\end{lem}
\begin{proof} If $y_1\le x_1$, the inequality is obvious. For $y_1>x_1$ we have $y_1\ge y_1-x_1>0$ and hence 
$\ga y_1^{\ga-1}\le\ga(y_1-x_1)^{\ga-1}$ so that
\begin{align*}
y_1^\ga-x_1^\ga\le\ga\int_{x_1}^{y_1}s^{\ga-1}~ds\le\ga\int_{x_1}^{y_1}(s-x_1)^{\ga-1}~ds=(y_1-x_1)^\ga.
\end{align*}
\end{proof}
\begin{prop}\label{prop:A8} Let $II=\dD \setminus I$ with $I$ as in proposition \ref{prop:A3}. Then
\begin{align*}
\left| \int_{II}\frac{\d G_i^k}{\d x_j}\om(y)~dy\right| \ale M x_i^{-\al}
\end{align*}
\end{prop}
\begin{proof}
Here we have to distinguish two cases. Assume first $d_1(x) = x_i$. Then
using proposition \ref{prop:A1}, \eqref{estOm} and Lemma \ref{lem:A2} we have
\begin{align*}
&\left| \int_{II}\frac{\d G_i^k}{\d x_j}\om(y)~dy\right| \ale \int_{II} \left|\frac{\d G_i^k}{\d x_j}\right||\om(y)|~dy\\
&\ale M \int_{II} |y-x|^{-3}y_i^{1-\al}~dy\ale M \int_{II} |y-x|^{-3}(|y_i-x_i|^{1-\al}+x_i^{1-\al})\\
&\ale M \int_{II} |y-x|^{-2-\al}~dy+M x_i^{1-\al}\int_{II} |y-x|^{-3}\\
&\ale M \int_{\half d_1(x)}^\infty\frac{1}{\rho^{2+\al}}~\rho~d\rho
+ M x_i^{1-\al}\int_{\half d_1(x)}^\infty\frac{1}{\rho^{3}}~\rho~d\rho\\
&\ale M d_1(x)^{-\al}+M x_i^{1-\al}d_1(x)^{-1}\ale M x_i^{-\al}+M x_i^{1-\al}x_i^{-1}\ale M x_i^{-\al}.
\end{align*}
Now let $d_1(x) = x_r$, $r\neq i$. Without loss of generality, we write down only the case $i=1$
(so $d_1(x)=x_2$). From
the explicit relations in Proposition \ref{propAllG} and the reflection inequalities we get
\begin{align*}
\left|\frac{\d G_1^k}{\d x_i}\right|\leq |y-x|^{-2} |y-\til x|^{-1} 
\end{align*}
and hence again by \eqref{estOm} and Lemma \ref{lem:A2},
\begin{align*}
&\left| \int_{II}\frac{\d G_1^k}{\d x_i}\om(y)~dy\right| \ale M \int_{II} \left|\frac{\d G_1^k}{\d x_i}\right| y_2^{1-\al}~dy\\
&\ale M \int_{II}|y-x|^{-1-\al} |y-\til x|^{-1} dy + M x_2^{1-\al}\int_{II} |y-x|^{-2}|y-\til x|^{-1}
\end{align*}
We continue with the integral without the factor $x_2^{1-\al}$ in front; first we enlarge the integration domain 
by replacing $II$ with $\dD$, then we split the integration domain into a ball $B(0, 2x_1)$ and the rest.
\begin{align*}
\int_{\dD}|y-x|^{-1-\al} |y-\til x|^{-1} dy \ale& \int_{\dD \cap B(0, 2 x_1)}|y-x|^{-1-\al} |y-\til x|^{-1} dy\\
&+ \int_{\dD \setminus B(0, 2 x_1)}|y-x|^{-1-\al} |y-\til x|^{-1} dy\\
 \ale& ~x_1^{-1}  \int_{\dD \cap B(x, 10 x_1)}|y-x|^{-1-\al} dy + \int_{B(0, 10) \setminus B(0, 2 x_1)}|y|^{-2-\al} dy\\
\ale& ~x_1 x_1^{1-\al} + x_1^{-\al} \ale x_1^{-\al}. 
\end{align*} 
Here, we used that $x_2\leq x_1$ (since $d_1(x)=x_2$), so that $|y|\ale |y-x|, |y|\ale |y-\til x|$ in the second in\-te\-gral.
Continuing with $x_2^{1-\al}\int_{II} |y-x|^{-2}|y-\til x|^{-1}$, we get
\begin{align*}
& x_2^{1-\al} \int_{II} |y-x|^{-2}|y-\til x|^{-1} dy \ale  x_2^{1-\al} x_1^{-\al} \int_{II} |y-x|^{-3+\al} dy\\
& \ale x_2^{1-\al} x_1^{-\al} x_2^{-1+\al} \al x_1^{-\al}
\end{align*} 
using $|y-\til x|\geq x_1^{\al} |y-\til x|^{1-\al}\geq x_1^{\al} |y-x|^{1-\al}$.
\end{proof}


\begin{prop}\label{prop:A9}
For $i\neq j$,
\begin{align*}
\left| \int_{\ov\setminus \dD}\left[\frac{\d G_i^1}{\d x_j}+\frac{\d G_i^2}{\d x_j}\right]\om(y)~dy\right| 
\leq ~~C(\ga_1,\ga_2)x_i^{-(\ga_1+\ga_2)}x_j^{\ga_2} d(x)^{-1+\ga_1}
\end{align*}
where $\ga_1\in (0, 1), \ga_2\in [0, 1), \ga_1+\ga_2 < 1$. Also,
\begin{align*}
\left| \int_{\ov\setminus \dD}\left[\frac{\d G_i^i}{\d x_1}+\frac{\d G_i^2}{\d x_i}\right]\om(y)~dy\right| 
\leq ~~C(\ga_1)x_i^{-\ga_1} d(x)^{-1+\ga_1}
\end{align*}
\end{prop}
\begin{proof}
As a preparation, we note that for $x\in D$, $0<\ga_1< 1$,
\begin{align}\label{prop_a9_ineq1}
\int_{\ov\setminus \dD} |y-x|^{-3 + \ga_1} \ale  d(x)^{-1+\ga_1}.
\end{align}
This follows from
\begin{align*}
\int_{\ov\setminus \dD} |y-x|^{-3 + \ga_1} &\leq\int_{\ov \setminus B(x, d(x))} |y-x|^{-3 + \ga_1}\\
&\leq \int_{B(x, 10)\setminus B(x, d(x))}|y-x|^{-3 + \ga_1},
\end{align*}
since $\ov\setminus \dD$ is contained in  $\ov \setminus B(x, d(x))$ because of $\de_2 < \de_3$ and $\de_1 < \de_2$.

From Proposition \ref{prop:A6} we get in case $i\neq j$
\begin{align*}
\left| \int_{\ov\setminus \dD}\left[\frac{\d G_i^1}{\d x_1}+\frac{\d G_i^2}{\d x_j}\right]\om(y)~dy\right| &\leq
x_i^{-(\ga_1+\ga_2)} x_j^{\ga_2} \int_{\ov\setminus \dD} |y-x|^{-3+\ga_1} \\
&\leq x_i^{-(\ga_1+\ga_2)} x_j^{\ga_2} d(x)^{-1+\ga_1},
\end{align*}
according to \eqref{prop_a9_ineq1}.

For the second inequality of the Proposition, we note that 
\begin{align*}
\left|\frac{\d G_i^k}{\d x_i}\right|&\ale x_i^{-\ga_1} |y-x|^{-3+\ga_1},
\end{align*}
and use \eqref{prop_a9_ineq1}.
\end{proof}


\begin{prop}\label{prop:A4}For $x\in D$,
\begin{align*}
\left|\pvint_{\ov}\left[\frac{\d G_i^1}{\d x_j}+\frac{\d G_i^2}{\d x_j}\right] \om(y)~dy\right|&\leq
 M x_i^{-\al} (1+\de_{j2}|\log d(x)|)\\
&\,\, +  C(\ga_1, \ga_2) x_i^{-(\ga_1+\ga_2)}x_j^{\ga_2} d(x)^{-1+\ga_1}  \quad (i\neq j)\\
\left|\pvint_{\ov}\left[\frac{\d G_i^1}{\d x_i}+\frac{\d G_i^2}{\d x_i}\right] \om(y)~dy\right|&\leq 
 M x_i^{-\al}(1+\de_{i2}|\log d(x)|) + C(\ga_1) x_i^{-\ga_1}d(x)^{-1+\ga_1}
\end{align*}
with $\ga,\ga_1\in (0, 1), \ga_2\in [0, 1), \ga_1+\ga_2 < 1$. 
\end{prop}
\begin{proof}
We split the integral into a principal value integral over $\dD$ and a convergent integral over $\ov \setminus \dD$. The integral over $\dD$ is 
further split in to integrals over the domains $I=B(x, \half d_1(x))$ and $II= \dD\setminus I$, which are estimated by Propositions
\ref{prop:A3} and \ref{prop:A8}. The part over $\ov \setminus \dD$ is estimated by Proposition \ref{prop:A9}.
\end{proof}


\vspace{0.5cm}
\subsection{Appendix B}~\\

\noindent
{\bf Derivation of $Q_i$ and $G_i^j$.}
\vspace{0.1cm}

\noindent
In all integrals over infinite domains it is understood that $\om$ is extended periodically and the integrals are understood as
limits in the mean.
We derive only $Q_1$ and the formulas for $G_1^1,~G_1^2$. $Q_2,~G_2^1$ and $G_2^2$ are analogous.
The first component of the velocity field is
\begin{align*}
u_1(x, t)& = \frac{1}{2\pi}\int_{\R^2} \frac{-(y_2-x_2)}{|y-x|^2} \om(y, t)~dy\\
&=\frac{1}{2\pi}\left\{\int_{(0,\infty)^2}
+\int_{(-\infty,0)\times (0,\infty)}+\int_{(0,\infty)\times (-\infty,0)}+\int_{(-\infty,0)^2}\right\} \frac{-(y_2-x_2)}{|y-x|^2} \om(y, t)~dy.
\end{align*}
Recall $\til x=(-x_1,x_2),~\bar x=(x_1,-x_2)$. Using the double-odd symmetry of $\om$ we can write
\begin{align*}
\int_{(-\infty,0)\times (0,\infty)}\frac{-(y_2-x_2)}{|y-x|^2} \om(y, t)~dy&=\int_{(0,\infty)^2}\frac{y_2-x_2}{|y-\til x|^2} \om(y, t)~dy\\
\int_{(0,\infty)\times (-\infty,0)}\frac{-(y_2-x_2)}{|y-x|^2} \om(y, t)~dy&=\int_{(0,\infty)^2}\frac{-(y_2+x_2)}{|y-\bar x|^2} \om(y, t)~dy\\
\int_{(-\infty,0)^2} \frac{-(y_2-x_2)}{|y-x|^2} \om(y, t)~dy&=\int_{(0,\infty)^2}\frac{y_2+x_2}{|y+x|^2} \om(y, t)~dy.
\end{align*}
Next we group the integrals in the following way
\begin{align*}
&\int_{(0,\infty)^2} \frac{-(y_2-x_2)}{|y-x|^2} \om(y, t)~dy + \int_{(0,\infty)^2}\frac{y_2-x_2}{|y-\til x|^2} \om(y, t)~dy\\
=&-4x_1\int_{(0,\infty)^2} \underbrace{\frac{y_1(y_2-x_2)}{|y-x|^2|y-\til x|^2}}_{G_1^1(x,y)} \om(y, t)~dy
\end{align*}
\begin{align*}
&\int_{(0,\infty)^2}\frac{-(y_2+x_2)}{|y-\bar x|^2} \om(y, t)~dy+\int_{(0,\infty)^2}\frac{y_2+x_2}{|y+x|^2} \om(y, t)~dy\\
=&-4x_1\int_{(0,\infty)^2} \underbrace{\frac{y_1(y_2+x_2)}{|y+x|^2|y-\bar x|^2}}_{G_1^2(x,y)} \om(y, t)~dy.
\end{align*}
Together we obtain
\begin{align*}
u_1(x,y)=&-x_1\frac{2}{\pi}\int_{[0,1]^2}[G_1^1(x,y)+G_1^2(x,y)]\om(y, t)~dy\\
&-x_1\frac{2}{\pi}\int_{\R_+^2\setminus [0,1]^2}[G_1^1(x,y)+G_1^2(x,y)]\om(y, t)~dy\\
=&-x_1\left(\frac{2}{\pi}\int_{[0,1]^2}[G_1^1(x,y)+G_1^2(x,y)]\om(y, t)~dy+ Q_1^r(x,t)\right)=-x_1Q_1(x,t).
\end{align*}

\begin{prop}\label{propAllG} The following relations hold:
\begin{align*}
\frac{\d G_1^1}{\d x_1}&=-\frac{2y_1(y_1+x_1)(y_2-x_2)}{|y-x|^2|y-\til x|^4}
+\frac{2y_1(y_1-x_1)(y_2-x_2)}{|y-x|^4|y-\til x|^2}\\
\frac{\d G_1^1}{\d x_2}&=\frac{2y_1(y_2-x_2)^2}{|y-x|^2|y-\til x|^4}
+\frac{2y_1(y_2-x_2)^2}{|y-x|^4|y-\til x|^2}-\frac{y_1}{|y-x|^2|y-\til x|^2}\\
\frac{\d G_1^1}{\d y_1}&=-\frac{2y_1(y_1+x_1)(y_2-x_2)}{|y-x|^2|y-\til x|^4}
-\frac{2y_1(y_1-x_1)(y_2-x_2)}{|y-x|^4|y-\til x|^2}+\frac{y_2-x_2}{|y-x|^2|y-\til x|^2}\\
\frac{\d G_1^1}{\d y_2}&=-\frac{2 y_1(y_2-x_2)^2}{|y-x|^2|y-\til x|^4}
-\frac{2y_1(y_2-x_2)^2}{|y-x|^4|y-\til x|^2}+\frac{y_1}{|y-x|^2|y-\til x|^2}
\end{align*}
\begin{align*}
\frac{\d G_1^2}{\d x_1}&=-\frac{2y_1(y_1+x_1)(y_2+x_2)}{|y+x|^4|y-\bar x|^2}
+\frac{2y_1(y_1-x_1)(y_2+x_2)}{|y+x|^2|y-\bar x|^4}\\
\frac{\d G_1^2}{\d x_2}&=-\frac{2y_1(y_2+x_2)^2}{|y+x|^4|y-\bar x|^2}
-\frac{2y_1(y_2+x_2)^2}{|y+x|^2|y-\bar x|^4}+\frac{y_1}{|y+x|^2|y-\bar x|^2}\\
\frac{\d G_1^2}{\d y_1}&=-\frac{2y_1(y_1+x_1)(y_2+x_2)}{|y+x|^4|y-\bar x|^2}
-\frac{2y_1(y_1-x_1)(y_2+x_2)}{|y+x|^2|y-\bar x|^4}+\frac{y_2+x_2}{|y+x|^2|y-\bar x|^2}\\
\frac{\d G_1^2}{\d y_2}&=-\frac{2y_1(y_2+x_2)^2}{|y+x|^4|y-\bar x|^2}
-\frac{2y_1(y_2+x_2)^2}{|y+x|^2|y-\bar x|^4}+\frac{y_1}{|y+x|^2|y-\bar x|^2}
\end{align*}
\begin{align*}
\frac{\d G_2^1}{\d x_1}&=-\frac{2y_2(y_1+x_1)^2}{|y+x|^4|y-\til x|^2}
-\frac{2y_2(y_1+x_1)^2}{|y+x|^2|y-\til x|^4}+\frac{y_2}{|y+x|^2|y-\til x|^2}\\
\frac{\d G_2^1}{\d x_2}&=-\frac{2y_2(y_1+x_1)(y_2+x_2)}{|y+x|^4|y-\til x|^2}
+\frac{2y_2(y_1+x_1)(y_2-x_2)}{|y+x|^2|y-\til x|^4}\\
\frac{\d G_2^1}{\d y_1}&=-\frac{2y_2(y_1+x_1)^2}{|y+x|^4|y-\til x|^2}
-\frac{2y_2(y_1+x_1)^2}{|y+x|^2|y-\til x|^4}+\frac{y_2}{|y+x|^2|y-\til x|^2}\\
\frac{\d G_2^1}{\d y_2}&=-\frac{2y_2(y_1+x_1)(y_2+x_2)}{|y+x|^4|y-\til x|^2}
-\frac{2y_2(y_1+x_1)(y_2-x_2)}{|y+x|^2|y-\til x|^4}+\frac{y_1+x_1}{|y+x|^2|y-\til x|^2}
\end{align*}
\begin{align*}
\frac{\d G_2^2}{\d x_1}&=\frac{2y_2(y_1-x_1)^2}{|y-x|^2|y-\bar x|^4}
+\frac{2y_2(y_1-x_1)^2}{|y-x|^4|y-\bar x|^2}-\frac{y_2}{|y-x|^2|y-\bar x|^2},\\
\frac{\d G_2^2}{\d x_2}&=-\frac{2y_2(y_1-x_1)(y_2+x_2)}{|y-x|^2|y-\bar x|^4}
+\frac{2y_2(y_1-x_1)(y_2-x_2)}{|y-x|^4|y-\bar x|^2},\\
\frac{\d G_2^2}{\d y_1}&=-\frac{2y_2(y_1-x_1)^2}{|y-x|^2|y-\bar x|^4}
-\frac{2y_2(y_1-x_1)^2}{|y-x|^4|y-\bar x|^2}+\frac{y_2}{|y-x|^2|y-\bar x|^2},\\
\frac{\d G_2^2}{\d y_2}&=-\frac{2y_2(y_1-x_1)(y_2+x_2)}{|y-x|^2|y-\bar x|^4}
-\frac{2y_2(y_1-x_1)(y_2-x_2)}{|y-x|^4|y-\bar x|^2}+\frac{y_1-x_1}{|y-x|^2|y-\bar x|^2}.
\end{align*}
\end{prop}

\bibliographystyle{amsplain}

\end{document}